\documentclass[11pt, a4paper]{amsart}
\usepackage{amsmath}
\usepackage{amssymb}
\usepackage{color}
\usepackage{amsthm}
\usepackage{dsfont}
\usepackage{amsfonts}%
\usepackage{color}
\usepackage{graphicx}
\usepackage{hyperref}

\newtheorem{Theorem}{Theorem}[section]
\newtheorem{Lemma}{Lemma}[section]

\newtheorem{Corollary}{Corollary}[section]

\allowdisplaybreaks

\begin{document}

\title[Travelling-wave behaviour in doubly nonlinear reaction-diffusion equations]{Travelling-wave behaviour in doubly nonlinear reaction-diffusion equations}

\author[Y.\,Du, A.\,G\'arriz \and F.\,Quir\'os]{Yihong Du, Alejandro G\'arriz, and Fernando Quir\'os}

\address{Yihong Du\hfill\break\indent
	School of Science and Technology, University of New England,
	\hfill\break\indent Armidale, NSW 2351, Australia.}\email{{\tt
	ydu@une.edu.au}}

\address{Alejandro G\'arriz\hfill\break\indent
	Instituto de Ciencias Matem\'aticas ICMAT (CSIC-UAM-UCM-UC3M),
	\hfill\break\indent 28049-Madrid, Spain,
		\hfill\break\indent and
	Departamento  de Matem\'{a}ticas, Universidad Aut\'{o}noma de Madrid,
	\hfill\break\indent 28049-Madrid, Spain.} \email{{\tt
		alejandro.garriz@estudiante.uam.es} }

\address{Fernando Quir\'{o}s\hfill\break\indent
	Departamento  de Matem\'{a}ticas, Universidad Aut\'{o}noma de Madrid,
	\hfill\break\indent 28049-Madrid, Spain,
	\hfill\break\indent and Instituto de Ciencias Matem\'aticas ICMAT (CSIC-UAM-UCM-UC3M),
	\hfill\break\indent 28049-Madrid, Spain.} \email{{\tt
		fernando.quiros@uam.es} }

\thanks{Y.~Du was supported by the Australian Research Council. A.~G\'arriz and F.~Quir\'os were supported by the Spanish Ministry of Science and Innovation,  through projects  MTM2017-87596-P and SEV-2015-0554, and by the Spanish National Research Council, through project 20205CEX001.}





\date{}

\begin{abstract}
	We study a family of reaction-diffusion equations that present a doubly nonlinear character given by a combination of the $p$-Laplacian and the porous medium operators. We consider the so-called slow diffusion regime, corresponding to a degenerate behaviour at the level 0, \normalcolor in which nonnegative solutions with compactly supported initial data have a compact support for any later time. For some results we will also require $p\ge2$ to avoid the possibility of a singular behaviour away from 0.
	
	Problems in this family have a unique (up to translations) travelling wave with a finite front. When the initial datum is bounded, radially symmetric and compactly supported, we will prove that solutions converging to 1 (which exist, as we show, for all the reaction terms under consideration for  wide classes of initial data) do so by approaching a translation of this unique traveling wave in the radial direction, but with a logarithmic correction in the position of the front when the dimension is bigger than one. As a corollary we obtain the asymptotic location of the free boundary and level sets in the non-radial case up to an  error term of size $O(1)$.  In dimension one we extend our results to cover  the case of non-symmetric initial data, as well as the case of bounded initial data with supporting sets unbounded in one direction of the real line. 	A main technical tool of independent interest is an estimate for the flux.
	
	
	Most of our results are new even for the special cases of the porous medium equation and the $p$-Laplacian evolution equation.
\end{abstract}

\maketitle

\section{Introduction}
\label{sect-introduction} \setcounter{equation}{0}

The aim of this paper it to characterize the large time behaviour of solutions to the Cauchy problem
\begin{equation}\label{eq:main}
u_t = \Delta_p u^m + h(u)\quad\text{in }Q:=\mathbb{R}^N\times\mathbb{R}_+,  \quad u(\cdot,0)=u_0 \quad\text{in }\mathbb{R}^N,
\end{equation}
where $\Delta_p$ stands for the well-known $p$-Laplacian operator,
$$
\Delta_p u:= \operatorname{div} (|\nabla u|^{p-2} \nabla u).
$$
The initial datum $u_0\not\equiv0$ is assumed to be bounded, nonnegative and compactly supported.

The nonlinearity $h$ is assumed to be in $C^1(\overline{\mathbb{R}}_+)$ and to fulfill, for some $a\in[0,1)$,
\begin{equation}\label{eq:reaction}
h(0)=0, \quad
h(u)\leq 0\text{ if } u \in [0,a],\quad h(u)>0\text{ if } u \in (a,1),\quad h(u)<0\text{ if } u > 1.
\end{equation}
We also ask our reaction term to be nondegenerate at $u=1$, that is,
\begin{equation}\label{eq:reaction_deriv_in_1}
h^\prime (1)  <0.
\end{equation}
If $a=0$, we are in the so-called \emph{monostable} case, which includes the Pearl-Verhulst (or logistic) reaction nonlinearity $h(u)=u(1-u)$. The case $a>0$ contains as particular instances the classical \emph{bistable} nonlinearities, when $h(u)<0$ for $u\in(0,a)$,  and the \emph{combustion} ones, when $h(u)=0$ for $u\in[0,a]$ (in the applications of this case $a$ corresponds to the ignition temperature).

As for the parameters, we will restrict ourselves to the so-called \textit{slow diffusion regime},
\begin{equation}\label{eq:slow.regime}\tag{SDR}
m > 0,\quad p>1,\quad m(p-1)>1.
\end{equation}
In this regime the operator degenerates at the level 0, and solutions propagate with finite speed: if the initial function is compactly supported, the same is true for $u(\cdot,t)$ for all $t>0$; see~\cite{Gilding-Kersner-1996}.  Thus, solutions to problem~\eqref{eq:main} have a \emph{free boundary},
separating the positivity region of $u$ from the region where $u$ vanishes. One of the goals of the present paper is  to determine its location for large times.  At some points, in addition to~\eqref{eq:slow.regime} we will also require  $p \geq 2$, to avoid the possibility of  a singular behaviour of the operator away from $u=0$.

The equation in~\eqref{eq:main} is referred to as being doubly nonlinear because of its nonlinearity in both $u$ and $\nabla u$.  For $p = 2$ and $m >1$ it is widely known as the porous medium equation, and  for $m=1$ and $p>2$ as the $p$-Laplacian evolution equation. Most of our results are new even for these two special cases.

Problems of the family~\eqref{eq:main} are used as models to describe the spreading of biological or chemical species, where the free boundary represents the expanding front, beyond which the species cannot be observed. They also have applications to describe several physical situations. See for instance~\cite{Esteban-Vazquez-1988,Gurtin-MacCamy-1977,Kalashnikov-1987,Kanel-1964,SanchezGarduno-Maini-1994,Zeldovich-1948} and the references therein.  This kind of models were first introduced in 1937 independently by Fisher~\cite{Fisher-1937},  and by Kolmogorov, Petrovsky and Piscounov~\cite{Kolmogorov-Petrovsky-Piscounov-1937} in the semilinear case $m=1$, $p=2$, with a monostable nonlinearity, $a=0$,  to study spreading questions in population genetics.

For all types of the reaction terms under consideration, which are quite general, we will give conditions on the initial data guaranteeing that the corresponding solution converges to 1 uniformly in compact sets as time goes to infinity, a situation that, following the literature, we will call \emph{spreading}. As we will see, for  certain monostable reaction nonlinearities all nontrivial nonnegative solutions spread. However, for bistable and combustion nonlinearities there are also  solutions converging uniformly to 0, a situation we call \emph{vanishing}. These results were proved  in the semilinear case  by Aronson and Weinberger in their classical paper~\cite{Aronson-Weinberger-1978}, and by one of the authors in the porous medium case   in the recent paper~\cite{Garriz-2020}. For the general doubly nonlinear problem the only available information is due to Audrito  who shows the existence of both spreading and vanishing solutions  for bistable nonlinearities in~\cite{Audrito-2019}.

The problems that we are considering have a unique travelling wave with a monotonic front connecting 1 to 0 and supported in $(-\infty,0]$; see for instance~\cite{Gilding-Kersner-2004} for the case $p=2$ and~\cite{Audrito-2019,Audrito-Vazquez-2017,Garriz-plap} for $p\neq 2$.
For $p\ge 2$ we will prove that the free boundary of any \emph{spreading} radial solution with  bounded and compactly supported initial datum approaches as $t\to\infty$ the sphere of radius
$$
c^*t-(N-1)c_\#\log t-r_0,
$$
where $c^*$ is the velocity of the unique travelling wave mentioned above, $c_\#>0$ is a constant independent of the solution, and $r_0\in \mathbb{R}$ is a constant that depends on the initial datum.
Moreover, in the moving radial coordinate in which the free boundary is fixed,  such solutions converge towards the profile of the aforementioned travelling wave, which implies that level sets of height $\lambda\in(0,1)$ move asymptotically like the front, though away from the free boundary by some constant distance $r_0(\lambda)$ that depends not only on the initial datum, but also on $\lambda$. As a corollary we obtain the location of the free boundary and of level sets for nonradial spreading solutions with compactly supported initial data up to an $O(1)$ error term. These are the main results of the paper, which can be slightly improved in the one-dimensional case.

In the semilinear case similar results hold when $a>0$; see~\cite{Fife-McLeod-1975,Kanel-1962,Uchiyama-1985}.  When $a=0$ one has to distinguish between the so called \emph{pushed} case, $c^* >
\sqrt{2h'(0)}$, for which we have also analogous results~\cite{Stokes-1976,Uchiyama-1985},  and  the more involved \emph{pulled} case, $c^*=\sqrt{2h'(0)}$, for which there is a logarithmic term (known as Bramson's correction term) in the description of the large-time behaviour of level sets even for $N=1$~\cite{Bramson-1983,Ducrot-2015,Gartner-1982,Hamel-Nolen-Roquejoffre-Ryzhik-2013,Roquejoffre-Rossi-RoussierMichon-2019,Rossi-2017}. In all cases the first term in this description had been given in~\cite{Aronson-Weinberger-1978}.

In the case of general doubly nonlinear operators, the first term of the asymptotic behaviour of the front was known both for  concave monostable nonlinearities~\cite{Audrito-Vazquez-2017} and bistable ones~\cite{Audrito-2019}. As for the existence of a logarithmic correction and the convergence to a travelling wave profile, there  are only precedents for the porous medium case. The first one~\cite{Biro-2002} deals with the particular reaction function $h(u)=u^p-u^q$, with $p< \min\{m,q\}$ in dimension $N=1$; see also~\cite{Kamin-Rosenau-2004} for the case $p=1$, $q=m$. More general reaction  nonlinearities, but still in dimension one, are considered in~\cite{Garriz-2020}. The only known results for higher dimensions come from~\cite{Du-Quiros-Zhou-Preprint}, but they are restricted to the logistic reaction nonlinearity. The present paper borrows some ideas from~\cite{Du-Quiros-Zhou-Preprint}. However,  there are  important differences, arising on the one hand from the extra degeneracy when the gradient vanishes, and on the other hand from the fact that we are considering much more general nonlinearities.

Let us remark that, in contrast with the semilinear case, in the slow diffusion regime problem~\eqref{eq:main} never has a logarithmic correction in dimension $N=1$. The same is true for the Stefan problem with reaction nonlinearities satisfying~\eqref{eq:reaction}--\eqref{eq:reaction_deriv_in_1},  another interesting free boundary problem for which results analogous to ours were proved in~\cite{Du-Matsuzawa-Zhou-2015}, with the extra nondegeneracy assumption $h'(0)>0$ if $a=0$.

When $N>1$, our proof requires on the one hand a careful (and quite involved) analysis of how convection affects  the velocity and the profile of the travelling  wave with a finite front, and on the other hand a bound for the \emph{flux} $|\nabla u^m|^{p-2}\nabla u^m$, which was only available for certain monostable nonlinearities and $p=2$. This bound, which has independent interest, is maybe the main technical novelty of the paper.

Another interesting technical difference arises in a preliminary estimate showing that spreading solutions approach~1 exponentially in compact sets when $p>2$. The proof for $p=2$ and logistic reaction nonlinearity given in~\cite{Du-Quiros-Zhou-Preprint} uses a result from~\cite{Du-Polacik-2015} for the semilinear case and a linearization argument. This idea does not work for $p>2$, and we have to develop a completely different proof based on a comparison argument with an explicit subsolution. This new approach depends in an essential way on the nature of the diffusion operator, and does not apply to the case $p=2$. When $p\in(1,2)$ spreading solutions are not expected to approach 1 exponentially, but at most in a power-like manner. This can be proved easily for many initial data by means of comparison with travelling waves; see the remark after Lemma~\ref{lem:behaviour.pressure.infinity} below.

\section{Preliminaries and main results}
\label{sect-prelim.main} \setcounter{equation}{0}

This section is devoted to give some preliminary results for both general and travelling wave solutions to~\eqref{eq:main}, and to describe the main  results of the paper and the key ideas behind them.

\subsection{Well-posedness}
The equation in~\eqref{eq:main} is degenerate either when $u=0$ or $\nabla u =0$, and does not have in general classical solutions.  A function $u$  is a \emph{weak solution} to problem~\eqref{eq:main} if $u,|\nabla u^m|^{p-1}, h(u)\in L^1_{\rm loc}(Q)$, and
\begin{equation}
\label{eq:weak.solution}
\int_{\mathbb{R}^N}u_0\varphi(\cdot,0)+\int_0^t\int_{\mathbb{R}^N}
\big(u\varphi_t-|\nabla u^m|^{p-2}\nabla u^m\cdot\nabla\varphi+h(u)\varphi\big)=0\quad\text{for each }\varphi\in C_{\rm c}^\infty(\overline Q).
\end{equation}
 If the equality in~\eqref{eq:weak.solution} is replaced by \lq\lq$\ge$'' (respectively, by \lq\lq$\le$'') for $\varphi\ge 0$, we have a subsolution (respectively, a supersolution).

Results about existence of weak solutions for the purely diffusive problem and its generalizations can be
found in the survey~\cite{Kalashnikov-1987} and the large number of references therein. As for the complete reaction-diffusion problem, since $h$ is locally Lipschitz, existence can be easily proved without much effort following the lines of what is done for the case $p=2$ in~\cite{Sacks-1983}, once one has an estimate for $|\nabla u^m|$ for a family of approximate problems, and using also H\"older regularity results from~\cite{Porzio-Vespri-1993,Vespri-1992} to obtain compactness. The solution obtained by this procedure is  continuous in $Q$.  The required estimate for $|\nabla u^m|$, which will also be used in the study of the asymptotic behaviour, is obtained in Section~\ref{subsect-bound.flux}.

Uniqueness can be proved as in~\cite{Li-2001}, and a comparison principle for sub- and supersolutions following the ideas in~\cite{diBenedetto-book,Wu-Zhao-Yin-Li-book,Zhao-Xu-SCh-1996}; see also~\cite{Alt-Luckhaus-1983}.

\subsection{Travelling waves}
\label{subsect:TWs}
The large time behaviour of spreading solutions will be given in terms of \emph{travelling wave} solutions. By this we mean solutions of the form $\bar u_c(x,t)=U_c(x-ct)$ for some \emph{speed}~$c$ and  \emph{profile}  $U_c$ (depending on $c$), which should satisfy $U_c\in C(\mathbb{R})$, $|(U^m_c)^\prime|^{p-1},h(U)\in L^1_{\rm loc}(\mathbb{R})$,  and
$$
\displaystyle\int_{\mathbb{R}} \left(|(U^m_c)^\prime|^{p-2}(U^m_c)^\prime\varphi^{\prime} +cU_c  \varphi^\prime - h(U_c)\varphi\right) =0 \quad\text{for all }\varphi \in C_{\rm c}^{\infty} (\mathbb{R}).
$$

Two travelling wave profiles will be regarded as the same if they are a translation of each other. A monotonic travelling wave profile is called a \emph{wavefront} (from 1 to 0) if it connects the two equilibrium states 1 and 0, that is,
$$
\lim\limits_{\xi \rightarrow - \infty} U_{c}(\xi) = 1,\qquad\lim\limits_{\xi \rightarrow  \infty} U_{c}(\xi) = 0.
$$
For this type of profiles it is clear that $U_c'\leq 0$ where $U_c$ is positive and hence smooth, and thus
$|(U^m_c)^\prime|^{p-2}(U^m_c)^\prime = -|(U_c^m)'|^{p-1}$.

If the reaction  term in~\eqref{eq:reaction} falls in the case $a=0$,  there exists a minimal speed $c^* = c^*(m,p,h)>0$ such that  equation~\eqref{eq:main}  has a unique  wavefron $U_c$ for all $c \geq c^*$, but none for $c < c^*$.

If $a>0$ there is a unique speed $c^*=c^*(m,p,h)$ for which the equation has a wavefront. The sign of this speed matches the sign of
 $\int_0^1 h(u)u^{m-1}\,\textrm{d}u$. Since we are interested in how the level $u=1$ invades the whole space,
 in the sequel we will always assume that this integral has positive sign.

\noindent\emph{Notation. } In what follows $\alpha:=1/(p-1)$. Note that in the slow diffusion regime $\alpha\in(0,m)$.

The wavefront  for the speed $c^*$ satisfies $U_{c^*}<1$ (this is guaranteed by condition~\eqref{eq:reaction_deriv_in_1}; see~\cite{Garriz-plap}) and is finite  (we are using  the slow diffusion regime assumption here);  that is, there exists a value~$\xi_0$ such that $U_{c^*}(\xi) =0$ for all $\xi \geq \xi_0$ and $U_{c^*}(\xi) >0$ for all $\xi < \xi_0$. Moreover, $U^\prime_{c^*}<0$ for all $\xi < \xi_0$, $|(U_{c^*}^m)'|^{p-1} \in C(\mathbb{R})$ and $(U_{c^*}^m)^\prime (\xi) \rightarrow 0$ as $\xi \rightarrow -\infty$. In addition,
\begin{equation}\label{profile:property}
\lim\limits_{\xi \rightarrow \xi_0^-} \left(\frac{m}{m-\alpha}U_{c^*}^{m-\alpha}\right)^{\prime}(\xi) = -(c^*)^\alpha.
\end{equation}

As mentioned above, when   $a=0$ there are also wavefronts for speeds $c>c^*$. However, they are not finite: they are positive in the whole $\mathbb{R}$. Let us remark that if $h$ is not smooth at the origin, this positivity result is not true in general~\cite[Section 10]{Audrito-Vazquez-2017}.

Further information about these results on wavefronts and many others can be found in~\cite{Gilding-Kersner-2004} for the case $p=2$ and in~\cite{Audrito-2019,Audrito-Vazquez-2017,Enguica-Gavioli-Sanchez-2013,Garriz-plap} for $p\neq 2$.

\noindent\emph{Notation. } From now on $U_{c^*}$ will denote the unique travelling wave  profile with speed $c^*$ and   $\xi_0=0$, that is, with support $\overline{\mathbb{R}_-}$.

\subsection{Spreading}

Our first task is to extend  the results of~\cite{Aronson-Weinberger-1978} concerning spreading from the semilinear case
to the whole slow diffusion regime. This is the content of Section~~\ref{sect-propagation.vs.vanish}.


Our first result shows that for all the reaction nonlinearities that we are considering there are initial data for which spreading happens.
\begin{Theorem}
Let~\eqref{eq:slow.regime} hold and $u$ be a nonnegative solution of \eqref{eq:main}.
There exists a three-parameter family of continuous, bounded and compactly supported functions $\mathfrak{u}_0(\cdot; \rho,c,\eta)$ (see Section~\ref{sect-propagation.vs.vanish} for a precise description) such that  if
$u(x,0) \geq \mathfrak{u}_0(x- x_0;\rho,c,\eta)$
for some $x_0 \in \mathbb{R}^N$, and admissible $\rho, c,\eta>0$, then $u$ converges to 1 uniformly on compact sets as $t\to\infty$.
\end{Theorem}

It turns out that for certain monostable nonlinearities $h$ nontrivial solutions always spread, independently of the initial datum. Following~\cite{Aronson-Weinberger-1978}, this phenomenon is named as the \textit{hair-trigger effect}.
\begin{Theorem}
\label{thm:spreading.vs.vanishing}
Let \eqref{eq:slow.regime} hold and let $h$ satisfy~\eqref{eq:reaction} with $a=0$ and
$$
\displaystyle\liminf\limits_{u \rightarrow 0} \frac{h(u)}{u^{m(p-1) + p/N}} > 0.
$$
If $u$ is a nontrivial nonnegative solution of \eqref{eq:main}, then $u$ converges to 1 uniformly on compact sets  as $t\to\infty$.
\end{Theorem}

Let us remark that if $h(u)\leq ku^q$, with $q> q_F:=m(p-1)+ p/N$ and $k>0$, comparison with the problem with reaction term $ku^q$ shows that for certain small initial data  solutions asymptotically vanish; see for instance~\cite{Galaktionov-Levine-1998}. The critical exponent $q_F$ is known as  the \emph{Fujita} exponent.

When both spreading and vanishing are possible, threshold behaviours, different from spreading and vanishing, may occur; see~\cite{Du-Matano-2010,Muratov-Zhong-2017,Polacik-2011} for the semilinear case. We do not pursue this interesting subject here.

Next we ask ourselves how fast is the spreading. If we move from a point $x_0\in \mathbb{R}^N$ in a certain direction with a slow speed $c$, in the limit $t\to\infty$ we will see only the value $u=1$, while if $c$ is too big we will surpass the free boundary of our solution, and thus we will only see the value $u=0$. It turns out that, with $c^*$ given above,  the former happens when $c<c^*$, while the latter occurs when $c>c^*$. In this sense, the speed $c^*$ is called the \emph{critical speed} or the \emph{spreading speed} of the problem.

\begin{Theorem}\label{thm:spreading}
Let~\eqref{eq:slow.regime} hold  and $u$ be a  nonnegative solution of \eqref{eq:main}.

\noindent\emph{(i)} If the initial datum is bounded and compactly supported,  given any $x_0\in\mathbb{R}^N$ and $c>c^*$ there is a value $T$ such that
$
u(x,t) = 0$ for $|x - x_0| \geq ct$, $t\ge T$.

\noindent\emph{(ii)}
If spreading happens, for any $c \in (0,c^*)$ we have
$\lim\limits_{t \rightarrow \infty} \min\limits_{|y - y_0| \leq ct} u(y,t) = 1$.
\end{Theorem}
This question has been recently analyzed in~\cite{Audrito-2019,Audrito-Vazquez-2017} for bistable and concave monostable reactions.

\subsection{Convergence to travelling waves}
The main goal of the paper is to obtain a precise description of the asymptotic behaviour of nonnegative radial spreading solutions to~\eqref{eq:main} with bounded and compactly supported initial data. In particular we will give a sharp description of their  positivity set.  As already mentioned, this behaviour will be given in terms of the wavefront with critical speed $c^*$. Our results can be used to significantly sharpen Theorem~\ref{thm:spreading}, revealing the precise logarithmic shift of the free boundary (and all other level sets).  This part  is restricted to parameters in the slow diffusion regime satisfying additionally $p\ge 2$.

Given a radially symmetric and compactly supported nonnegative initial function $u_0\not\equiv 0$, the solution of \eqref{eq:main} is also radial and compactly supported for any positive time, and satisfies the equation
\begin{equation}
\label{eq:radial.equation}
u_t=\big(|(u^m)_r|^{p-2}(u^m)_r\big)_r+\frac{N-1}{r}|(u^m)_r|^{p-2}(u^m)_r+h(u),
\end{equation}
where $r=|x|$  and, abusing notation, $u=u(r,t)$. Let
\begin{equation*}
\label{eq:def.h}
\eta(t)=\inf\{r>0: u(x,t)=0\mbox{ if }|x|>r\}.
\end{equation*}
It is easy to show that after some finite time $T$ the spatial support of $u(\cdot, t)$ for any later time is a ball of radius $\eta(t)$: $u(x,t)>0$ if $|x|<\eta(t)$, and $u(x,t)=0$ if $|x|\ge \eta(t)$ for all $t\ge T$.

From Theorem~\ref{thm:spreading} we  see that if $u$ spreads, then $\lim\limits_{t\to\infty}\frac{\eta(t)}t= c^*$. Thus, close to the free boundary,
\[
\frac{N-1}{r}\approx \gamma(t):=\frac{N-1}{c^*t},
\]
and $u$ solves approximately the nonlinear reaction-diffusion-convection equation
\begin{equation*}\label{eq:convection}
u_t=\big(|(u^m)_r|^{p-2}(u^m)_r\big)_r+\gamma(t)|(u^m)_r|^{p-2}(u^m)_r+h(u).
\end{equation*}
As we will see in Subsection~\ref{subsect-wavefronts.convection}, for every small \emph{constant} $\gamma\ge0$ there is a unique speed $c(\gamma)$  (depending also on $m$, $p$, and $h$, of course), such that the equation
\begin{equation}\label{eq:constant.convection}
u_t=\big(|(u^m)_r|^{p-2}(u^m)_r\big)_r+\gamma|(u^m)_r|^{p-2}(u^m)_r+h(u)
\end{equation}
has a finite wavefront. This wavefront is unique up to translations; see also~\cite{Gilding-Kersner-2004,Malaguti-Ruggerini-2010} for the porous medium case.

\noindent\emph{Notation. } The unique wavefront for~\eqref{eq:constant.convection} with speed $c(\gamma)$ and   $\xi_0=0$, that is, with support $\overline{\mathbb{R}_-}$, will be denoted by $U(\cdot;\gamma)$ from now on.

We will prove in Subsection~\ref{subsect-dependence.convection} that $c(\gamma)$ is smooth close to the origin, and we conjecture that
$$
\eta'(t)\approx c(\gamma(t))\approx c(0)+c'(0)\frac{N-1}{c^*t}
$$
for large times.
Since $c(0)=c^*$, we expect that
$$
\eta(t)\approx c^*t -(N-1)c_\#\log t \quad\text{as }t\to\infty,\text with  c_\#=-\frac{c'(0)}{c^*}>0.
$$
We will also prove in Subsection~\ref{subsect-dependence.convection} that  $U(\cdot;\gamma)$ approaches $U_{c^*}$ as $\gamma\to0^+$. Hence, we expect spreading solutions to approach $U_{ c^*}$ in a suitable moving coordinate system (moving as the free boundary) as $t\to\infty$. We will show that this is indeed the case  if the initial data are radially symmetric, bounded and compactly supported.
\begin{Theorem}
\label{thm:convergence.compact.support}
Assume~\eqref{eq:slow.regime} and $p\geq 2$. Let $u$ be a nonnegative spreading solution of~\eqref{eq:main} corresponding to a bounded, radially symmetric and compactly supported initial datum $u_0$, and let $\eta(t)$ be the function describing its interface. Then there is a constant $r_0$ such that
$$
\begin{cases}
\lim\limits_{t \rightarrow \infty} \sup\limits_{r\geq 0} |u(r,t) - U_{c^*}(r-c^*t+(N-1)c_\sharp\log t - r_0)| = 0, \\[8pt]
\lim\limits_{t\to\infty}\big[\eta(t)-c^*t + (N-1)c_\#\log t\big]=r_0.
\end{cases}
$$
\end{Theorem}
These results,  which are proved in Section~\ref{sect-convergence.compact.support},  can be slightly improved in the one-dimensional case to include  on the one hand non-symmetric initial data, and on the other hand  bounded initial data with finite support only in one direction; see Subsection~\ref{subsect-one.dimension}.

The proof uses a bound for $|\nabla u^m|$, obtained in Section~\ref{sect-bound.flux}, and comparison with sub- and supersolutions  of  problem~\eqref{eq:convection} constructed from wavefronts. The argument requires having a thorough knowledge of how $c(\gamma)$ and the associated wavefront constructed in Subsection~\ref{subsect-wavefronts.convection} vary with the convection parameter $\gamma$. This analysis is performed in Subsection~\ref{subsect-dependence.convection}.

Theorem~\ref{thm:convergence.compact.support} combined with a comparison argument allows to show easily that the free boundary (and other level sets) of non-radial spreading solutions to~\eqref{eq:main} is always within $O(1)$ distance from the moving sphere $|x|= c_*t-(N-1)c_\#\log t$, thus improving  Theorem~\ref{thm:spreading}.

\begin{Corollary}\label{cor:1}

Assume~\eqref{eq:slow.regime} and $p\geq 2$. Let $u$ be a spreading solution to~\eqref{eq:main} with a bounded, nonnegative and compactly supported initial datum.
	\begin{itemize}
		\item[(i)]  For any $t>0$, let  $\Gamma(t):=\partial\{x\in\mathbb{R}^N:u(x,t)>0\}$.  There exist $r_1, r_2\in\mathbb R$ and $T>0$ such that
	\[
	\Gamma(t)\subset
	\big\{x\in\mathbb R^N: r_1\leq |x|-c_* t+(N-1)c_\# \log t\leq r_2\big\}\quad\text{for all }t\ge T.
	\]
	\item[(ii)] For any $l\in (0,1)$ and $t>0$, let $E_l(t):=\{x\in\mathbb R^N: u(x, t)=l\}$. There exist $r^l_1, r^l_2\in\mathbb R$ and $T>0$  such that
	\[
	E_l(t)\subset
	\big\{x\in\mathbb R^N: r^l_1\leq |x|-c_* t+(N-1)c_\# \log t\leq r^l_2\big\}\quad\text{for all }t\ge T.
	\]
	\item[(iii)]
	$\displaystyle\lim_{t\to\infty}u(x,t)=1 \mbox{ uniformly in } \big\{|x|\leq c_*t-(N-1)c\log t\big\} \text{ for any }c>c_\#$.
\end{itemize}
\end{Corollary}
A question that naturally arises is whether properties (i) and (ii) are sharp or not,
that is, does the difference between the radii of the internal and external balls  tend
to 0 as $t\to\infty$? Of course, for this question to make sense we have to consider all possible balls, not just the ones centred at the origin.
We expect a negative answer, in view of the counterexamples for bistable and concave monostable nonlinearities in the semilinear case given in~\cite{Rossi-2017,Roussier-2004,Yagisita-2001}. The best we expect is to show that there exists a Lipschitz function $s^\infty$ defined on the  unit sphere
such that $u(x,t)$ approaches, as $t$ goes to infinity, the function
$$
U_{c^*}(|x|-c^*t+(N-1)c_\sharp\log t+s^\infty(x/|x|)),
$$
a result that has been proved very recently  in~\cite{Roquejoffre-Rossi-RoussierMichon-2019} for the semilinear case with logistic reaction nonlinearity.

\subsection{A bound for the flux}\label{subsect-bound.flux}
In order to prove Theorem~\ref{thm:convergence.compact.support} we need a bound for the flux $|\nabla u^m|^{p-2}\nabla u^m$, in order to get rid  of the convection term appearing in~\eqref{eq:radial.equation} when $N>1$, in the limit $t\to\infty$,  for $r$  in the range close to the front. This bound can be proved for $t>0$ any distance away from  $t=0$.
\begin{Theorem}
	\label{thm:bound.flux}
 Assume~\eqref{eq:slow.regime} and
let  $u$ be a  nonnegative solution of equation~\eqref{eq:main} with $u_0\in L^\infty (\mathbb{R}^N)$. For every fixed time $\tau>0$ there exists a positive constant $k$ depending on $\tau$, $m$, $p$,   the reaction function $h$, and $\|u_0\|_\infty$, such that
	$$
	|\nabla u^m(x,t)|\leq  k\quad\text{for all }x\in\mathbb{R}^N,\ t\ge\tau.
	$$
\end{Theorem}

\noindent\emph{Remark. } A similar result holds for equation~\eqref{eq:convection} when $\gamma$ is constant.

\medskip

This bound  follows from an estimate for the gradient of the  \emph{pressure}  function $v:= \frac {m}{m-\alpha}u^{m-\alpha}$, proved in Section~\ref{sect-bound.flux}, which  has independent interest. The pressure is expected to rule the advance of the free boundary; see~\eqref{profile:property} for the case of travelling wave solutions.

In the porous medium case with a logistic reaction term Theorem~\ref{thm:bound.flux} was proved in~\cite{Du-Quiros-Zhou-Preprint} using an estimate from below for the Laplacian of the pressure due to~\cite{Perthame-Quiros-Vazquez-2014}.
An estimate from below for the $p$-Laplacian of the pressure  function is available for the general doubly nonlinear diffusion case with no reaction in dimension one~\cite{Esteban-Vazquez-1988}, which could, possibly,  be extended to include problems with certain reaction nonlinearities.  If successful, this would allow us to obtain the required bound for the flux.  Unfortunately, even in the porous medium case the estimate for the Laplacian of the pressure is only known to be valid under certain hypotheses on the reaction nonlinearity that leave out most of the nonlinearities we are interested in, in particular bistable and combustion ones. Thus we need a different approach.

The method that we will employ here is inspired by the work of Bernstein~\cite{Bernstein-1938}, which has been fruitfully applied to deal with nonlinear \emph{nondegenerate parabolic} equations; see, for example, \cite{Ilin-Kalashnikov-Oleinik-2001} and
further references given there. Aronson was the first  to apply this approach to a degenerate equation in~\cite{Aronson-1969}, where he considers the porous medium equation without reaction in dimension one. Versions including \emph{absorption} terms are also available~\cite{Herrero-Vazquez-1987}.  Aronson's paper was extended in~\cite{Esteban-Vazquez-1986} to cover doubly nonlinear diffusion operators, but still in dimension one and without reaction. Here we will deal with general nonlinearities satisfying~\eqref{eq:reaction} in any dimension, which makes the analysis quite involved and technically cumbersome.

\medskip

\noindent\emph{Notation. } In what follows $\mathcal{L}$ will denote the parabolic operator  defined through
\begin{equation*}\label{eq:def.L}
\mathcal{L} u:= u_t -\Delta_p u^m- h(u).
\end{equation*}


\section{A bound for the flux}
\label{sect-bound.flux} \setcounter{equation}{0}

We devote this section to find an estimate for $|\nabla u^m|$ that will imply that this quantity is uniformly bounded for all $t\ge \tau>0$.  The estimate will follow from a bound for the derivative of the pressure $v$.  In order to obtain the latter estimate we will use an $N$-dimensional Bernstein-type method.

\begin{Lemma}\label{lema_bounded_derivative_pressure}
Assume~\eqref{eq:slow.regime}. Let $u$ be a bounded nonnegative solution of problem~\eqref{eq:main} and let $v:= \frac {m}{m-\alpha}u^{m-\alpha}$ be the corresponding pressure. Let $\mathcal H:=\sup_{[0,M]}|h'|$, with $M=\max\{1,\|u_0\|_{L^\infty(\mathbb{R}^N)}\}$. There exists a constant $C$  depending only on $\|u_0\|_\infty$ and $h$ such that
$$
|\nabla v|(x,t) \leq CM^{m-\alpha} \max\{(t^{-1}+\mathcal H)^{1/2}, (t^{-1}+\mathcal H)^{1/p}\}\quad\text{for all }(x,t)\in Q.
$$
\end{Lemma}

\begin{proof} We proceed by approximation.
Let $u_\varepsilon $ be the (smooth) solution of the uniformly parabolic problem
$$
u_t(x,t)=\operatorname{div}\left((|\nabla u^m|^{p-2}  + \varepsilon)\nabla u^m\right) + h(u), \qquad u(\cdot,0)=u_{0\varepsilon},
$$
where $\{u_{0\varepsilon}\}$ is a family of  smooth positive functions that converge uniformly in compact sets to $u_0$  as $\varepsilon\to0$, and such that $0<\varepsilon\leq u_{0\varepsilon}\leq \|u_0\|_{L^\infty(\mathbb{R}^N)}$. The above regularized problem has a comparison principle that allows to prove in particular that  $\varepsilon \leq u_\varepsilon (x,t) \leq M$. The pressure associated to $u_\varepsilon$, which we will denote $v$ instead of $v_\varepsilon$ for simplicity, satisfies
$$
v_t = (m-\alpha)v\operatorname{div}\left((|\nabla v|^{p-2}  + \varepsilon)\nabla v \right) + |\nabla v |^p + \frac{h(\Phi(v))}{\Phi^\prime(v)},
$$
where $\Phi(v ) = ((m-\alpha)v /m)^{1/(m-\alpha)}$ gives the density as a function of the pressure.

We define the function $F:\mathbb{R}\to \mathbb{R}$ as
\begin{equation*}
\label{eq:definition.F}
F(v):= \nabla v \operatorname{Hess}(v) (\nabla v)^T = \sum\limits_{i,j = 1}^{N}v_{x_i} v_{x_j}v_{x_i x_j},
\end{equation*}
where $A^T$ represents the transpose of the matrix $A$. Hence, thanks to the identity
\begin{equation}\label{eq:grad_norm}
\nabla\big(|\nabla f|^{\beta}\big) = \beta|\nabla f|^{\beta-2} \nabla f \operatorname{Hess}(f),
\end{equation}
we may rewrite the equation of the pressure as
$$
v_t = (m-\alpha) v\Big( (p-2)|\nabla v|^{p-4} F(v) +  (|\nabla v|^{p-2}  + \varepsilon)\Delta v \Big) + |\nabla v |^p + \frac{h(\Phi(v))}{\Phi^\prime(v)}.
$$

Let $S:= \frac {m}{m-\alpha}M^{m-\alpha}$. Then $\|v\|_{L^\infty(Q)}\le S$.  Let   $\theta: [0,\delta] \to [0,S]$, $\delta>0$,  be a smooth, strictly increasing and concave function that will be completely specified later. If we define $w:\mathbb{R}^N\to [0,\delta]$ through $v=\theta (w)$, it will satisfy the equation
\begin{equation}
\label{eq:press_theta}
\begin{aligned}
w_t =
& \,(m-\alpha)\theta\left((p-2)(\theta')^{p-5}|\nabla w|^{p-4}  F(\theta)  +  ((\theta')^{p-2}|\nabla w|^{p-2} + \varepsilon )  \left(\frac{\theta''}{\theta'}|\nabla w|^2 + \Delta w\right)\right)\\
&+ (\theta')^{p-1}|\nabla w|^p+\frac{h(\Phi(\theta))}{\Phi^\prime(\theta) \theta^\prime},
\end{aligned}
\end{equation}
where we have omitted the dependence of $\theta$ and its derivatives on $w$ for simplicity. An estimate for $|\nabla w|$  will yield an estimate for $|\nabla v|$.

We note  that
\begin{equation}
\label{eq:F.theta}
F(\theta)=(\theta')^2\Big( \theta'' \sum\limits_{i,j = 1}^{N}w_{x_i}^2 w_{x_j}^2  + \theta' \sum\limits_{i,j = 1}^{N}w_{x_i} w_{x_j}w_{x_i x_j}   \Big) = (\theta')^2\big( \theta'' |\nabla w|^4  + \theta' F(w)   \big).
\end{equation}

Since  $w\in C^2$,  $\operatorname{Hess}(w)$ is symmetric, and hence
$$
\begin{aligned}
\nabla(F(w))&=\sum\limits_{i,j=1}^N \nabla w_{x_i} w_{x_j} w_{x_i x_j}
+\sum\limits_{i,j=1}^N w_{x_i}\nabla w_{x_j}  w_{x_i x_j} +\sum\limits_{i,j=1}^N w_{x_i} w_{x_j}   \nabla w_{x_i x_j} \\
&=2 \nabla w \operatorname{Hess}(w)(\operatorname{Hess}(w))^T +  \nabla w  \sum\limits_{i=1}^N w_{x_i}\operatorname{Hess}(w_{x_i}).
\end{aligned}
$$
Therefore, using also~\eqref{eq:grad_norm}, we get
\begin{equation}
\label{eq:gradient.F.theta}
\begin{aligned}
\nabla( F(\theta)) =&\big( 2\theta' (\theta'')^2 + (\theta')^2\theta''' \big)|\nabla w|^4 \nabla w + 7 (\theta')^2\theta''|\nabla w|^2 \nabla w  \operatorname{Hess}(w)\\
&+(\theta')^3\Big(2 \nabla w \operatorname{Hess}(w)(\operatorname{Hess}(w))^T +  \nabla w  \sum\limits_{i=1}^N w_{x_i}\operatorname{Hess}(w_{x_i})\Big).
\end{aligned}
\end{equation}

At this stage we cannot guarantee that $|\nabla w|$ is bounded. Hence we will multiply it by a cut-off function  $\zeta \in C^\infty(Q)\cap C(\overline Q)$ such that $0\leq \zeta \leq 1$ and $\zeta(x,t) = 0$ if $t=0$ or $|x| \geq R >0$.  Let $Q_T:=\mathbb{R}^N\times(0,T]$.

We start with the case $p\ge2$. Assume that $z:=|\nabla w|^2\zeta^2$ achieves a non-trivial maximum in $\overline{Q_T}$ at some $(x_0, t_0)\in Q_T$. Notice that such point has to lie within $\operatorname{supp}(\zeta)$.
We want to estimate this maximum value. To this aim we take gradients in both sides of~\eqref{eq:press_theta}, and then  multiply the resulting equation from the right by $(\nabla w)^T\zeta^2$. Using~\eqref{eq:F.theta} and~\eqref{eq:gradient.F.theta}, we arrive at
\begin{equation}\label{eq:grad_press_theta}
\displaystyle 0=\sum_{i=0}^{10}\mathcal{A}_i,\quad \text{where}
\end{equation}
$$
\begin{aligned}
\mathcal{A}_0=&\,-\nabla w_{t}(\nabla w)^T \zeta^2,
\\
\mathcal{A}_1=&\,(m-\alpha)\Big((p-2)(\theta')^{p-2}|\nabla w|^{p-2} \big(\theta''|\nabla w|^4+\theta'F(w)\big)\\
&  +  \big((\theta')^{p-2}|\nabla w|^{p-2} + \varepsilon \big) \big(\theta''|\nabla w|^2 + \theta'\Delta w\big)|\nabla w|^2\Big)\zeta^2,\\
\mathcal{A}_2=&\,(m-\alpha)(p-2)(p-5)\theta(\theta')^{p-4}\theta''|\nabla w|^{p-2} \big(\theta''|\nabla w|^4+\theta'F(w)\big) \zeta^2, \\
\mathcal{A}_3=&\,  (m-\alpha)(p-2)(p-4)\theta  (\theta')^{p-3}|\nabla w|^{p-6}\big(\theta''|\nabla w|^4+\theta'F(w)\big)  F(w)\zeta^2,\\
\mathcal{A}_4=&\, (m-\alpha)  (p-2)\theta(\theta')^{p-4}|\nabla w|^{p+2} \big( 2 (\theta'')^2 + \theta'\theta''' \big)\zeta^2,\\
\mathcal{A}_5=&\, 7(m-\alpha) (p-2)\theta (\theta')^{p-3}\theta''|\nabla w|^{p-2} F(w)\zeta^2,\\
\mathcal{A}_6=&\, 2(m-\alpha) (p-2)\theta (\theta')^{p-5}|\nabla w|^{p-4} (\theta')^3\nabla w \operatorname{Hess}(w)(\nabla w\operatorname{Hess}(w))^T\zeta^2,\\
\mathcal{A}_7=&\, (m-\alpha)  (p-2)\theta(\theta')^{p-2}|\nabla w|^{p-4}   \nabla w  \sum\limits_{i=1}^N w_{x_i}\operatorname{Hess}(w_{x_i})(\nabla w)^T\zeta^2,\\
\mathcal{A}_8=&\, (m-\alpha)(p-2)\theta(\theta')^{p-4}|\nabla w|^{p-4}  \big(\theta''|\nabla w|^2 + \theta'\Delta w\big)\big(\theta''|\nabla w|^4  + \theta' F(w) \big)\zeta^2,\\
\mathcal{A}_9=&\, (m-\alpha)\theta  \big((\theta')^{p-2}|\nabla w|^{p-2} + \varepsilon \big) \Big(\Big( \frac{\theta''}{\theta'} \Big)' |\nabla w|^4 + 2 \frac{\theta''}{\theta'} F(w) + \nabla(\Delta w)(\nabla w)^T\Big)\zeta^2,\\
\mathcal{A}_{10}=&\, (p-1)(\theta')^{p-2}\theta'' |\nabla w|^{p+2} \zeta^2 + p(\theta')^{p-1}|\nabla w|^{p-2}F(w)\zeta^2 +  \frac{\partial}{\partial w} \Big(\frac{h(\Phi(\theta))}{\Phi^\prime(\theta) \theta^\prime} \Big)|\nabla w|^2\zeta^2.
\end{aligned}
$$
To proceed, we want to obtain bounds from above  for the terms $\mathcal{A}_i$, $i=0,\dots,10$, at $(x_0,t_0)$,
by expressions involving  $\nabla w$, but not directly involving higher order derivatives of $w$ or $w$ itself. So we are already done with~$\mathcal{A}_4$.

Since $z_t=2\nabla w_{t} (\nabla w)^T\zeta^2 + 2|\nabla w|^2\zeta\zeta_t$ and $z_t(x_0,t_0)\ge0$, we have
\begin{equation*}
\label{eq:time.derivative.at.maximum}
\displaystyle\mathcal{A}_0=-\nabla w_{t}(\nabla w)^T \zeta^2 \leq |\nabla w|^2\zeta\zeta_t  \quad\text{at }(x_0,t_0).
\end{equation*}

Using~\eqref{eq:grad_norm} we get $\nabla z = 2 \nabla w \operatorname{Hess}(w) \zeta^2 + 2|\nabla w|^2\zeta \nabla \zeta$. Thus, since  $\nabla z(x_0,t_0)=0$, we have
\begin{equation}
\label{eq:Hessian.F.maximum}
\operatorname{Hess}(w)=- \frac{(\nabla w)^T \nabla \zeta}{\zeta}\quad\text{at }(x_0,t_0),
\end{equation}
and hence
$$
F(w)= - \frac{|\nabla w|^2 \nabla \zeta(\nabla w)^T}{\zeta},\quad
\displaystyle \Delta w = \operatorname{Trace}(\operatorname{Hess}(w))= -\frac{\nabla \zeta(\nabla w)^T}{\zeta}\quad\text{at }(x_0,t_0).
$$
Therefore, at $(x_0,t_0)$ we have the identities
$$
\begin{aligned}
\mathcal{A}_1&=(m-\alpha)\big(  (p-1)(\theta')^{p-2}|\nabla w|^{p-2} + \varepsilon \big)\big(\theta''\zeta^2|\nabla w|^2  - \theta'\zeta\nabla \zeta(\nabla w)^T\big)|\nabla w|^2,
\\
\mathcal{A}_2&= (m-\alpha)(p-2)(p-5)\theta (\theta')^{p-4}\theta''|\nabla w|^{p}  \big(\theta''\zeta^2|\nabla w|^2  - \theta'\zeta\nabla \zeta(\nabla w)^T\big),
\\
\mathcal{A}_3&=- (m-\alpha)  (p-2)(p-4)\theta(\theta')^{p-3}|\nabla w|^{p-2}\left(\theta''\zeta|\nabla w|^2  - \theta' \nabla \zeta (\nabla w)^T\right) \nabla \zeta (\nabla w)^T,
\\
\mathcal{A}_5&=- 7(m-\alpha) (p-2)\theta(\theta')^{p-3}\theta''\zeta|\nabla w|^{p} \nabla \zeta(\nabla w)^T,\\
\mathcal{A}_6&=2(m-\alpha)(p-2)\theta (\theta')^{p-2}|\nabla w|^{p}  |\nabla \zeta|^2,
\\
\mathcal{A}_8&= (m-\alpha) (p-2)\theta(\theta')^{p-4} |\nabla w|^{p-2}\big(\theta''\zeta |\nabla w|^{2} - \theta' \nabla \zeta(\nabla w)^T  \big)^2,
\\
\mathcal{A}_{10}&=(p-1)(\theta')^{p-2}\theta'' \zeta^2|\nabla w|^{p+2} - p(\theta')^{p-1}\zeta|\nabla w|^{p}  \nabla \zeta(\nabla w)^T +  \frac{\partial}{\partial w} \left(\frac{h(\Phi(\theta))}{\Phi^\prime(\theta) \theta^\prime} \right)\zeta^2|\nabla w|^2.
\end{aligned}
$$

In order to estimate $\mathcal{A}_7$ we compute
$$
\begin{aligned}
\operatorname{Hess}(z)=&\operatorname{Hess}(|\nabla w|^2)\zeta^2 + 4\zeta\big(\nabla w\operatorname{Hess}(w) \big)^T \nabla\zeta + 4\zeta (\nabla\zeta)^T \nabla w \operatorname{Hess}(w) + |\nabla w|^2 \operatorname{Hess}(\zeta^2)\\
=&\, 2\sum\limits_{i=1}^N \left((\nabla w_{x_i})^T  \nabla w_{x_i} +  w_{x_i}\operatorname{Hess}(w_{x_i}) \right)\zeta^2  + 4\zeta\big(\nabla w\operatorname{Hess}(w) \big)^T \nabla\zeta \\
&+ 4\zeta (\nabla\zeta)^T \nabla w \operatorname{Hess}(w) + 2 \big((\nabla \zeta)^T  \nabla \zeta +  \zeta\operatorname{Hess}(\zeta) \big
)|\nabla w|^2.
\end{aligned}
$$
Hence, using the calculus identity $\sum\limits_{i=1}^N (\nabla w_{x_i})^T \nabla w_{x_i} = (\operatorname{Hess}(w))^T  \operatorname{Hess}(w)$ and~\eqref{eq:Hessian.F.maximum}, we get
$$
\operatorname{Hess}(z)=  2\zeta^2 \sum\limits_{i=1}^N w_{x_i}\operatorname{Hess}(w_{x_i}) - 2\big( 2(\nabla\zeta)^T\nabla\zeta - \zeta \operatorname{Hess}(\zeta) \big)|\nabla w|^2\quad\text{at }(x_0,t_0).
$$
Thus, since
$\operatorname{Hess}(z)$ at $(x_0,t_0)$ is semi-definite negative,
\begin{equation*}
\label{eq:gradient.Hessian}
\nabla w \left( \sum\limits_{i=1}^N w_{x_i}\operatorname{Hess}(w_{x_i}) \right) (\nabla w)^T\zeta^2 \leq  \nabla w \left( 2(\nabla\zeta)^T\nabla\zeta - \zeta \operatorname{Hess}(\zeta) \right) (\nabla w)^T |\nabla w|^2\quad\text{at }(x_0,t_0),
\end{equation*}
and therefore
$$
\mathcal{A}_7\le(m-\alpha)  (p-2)\theta(\theta')^{p-2}|\nabla w|^{p-2}   \nabla w\left( 2(\nabla\zeta)^T\nabla\zeta - \zeta \operatorname{Hess}(\zeta) \right)(\nabla w)^T\quad\text{at }(x_0,t_0).
$$

In order to estimate the remaining term, $\mathcal{A}_9$, we compute
$$
\begin{aligned}
\Delta z &=   \Delta (|\nabla w|^2) \zeta^2 + 2 \nabla(|\nabla w|^2) (\nabla\zeta^2)^T + |\nabla w|^2 \Delta \zeta^2\\
&=  2\sum_{i=1}^N \left(w_{x_i}\Delta w_{x_i} + |\nabla w_{x_i}|^2\right)\zeta^2 + 8\zeta \nabla w \operatorname{Hess}(w)(\nabla\zeta)^T + |\nabla w|^2\Delta \zeta^2,
\end{aligned}
$$
which combined with the calculus identity $\sum\limits_{i=1}^N |\nabla w_{x_i}|^2 = \operatorname{Trace}\left((\operatorname{Hess}(w))^T  \operatorname{Hess}(w)\right)$ yields
$$
\Delta z = 2\nabla(\Delta w)(\nabla w)^T\zeta^2 + 2\zeta^2\operatorname{Trace}\left((\operatorname{Hess}(w))^T  \operatorname{Hess}(w)\right)+ 8\zeta \nabla w \operatorname{Hess}(w)(\nabla\zeta)^T+ |\nabla w|^2\Delta \zeta^2.
$$
Therefore, since $\Delta z(x_0,t_0) \leq 0$, using~\eqref{eq:Hessian.F.maximum} we get
\begin{equation*}
\label{eq:gradient.laplacian}
\nabla(\Delta w)  (\nabla w)^T \zeta^2 \leq \big(  3|\nabla \zeta|^2  - \Delta(\zeta^2/2)\big)|\nabla w|^2 \quad\text{at }(x_0,t_0).
\end{equation*}
Thus, using also once more~\eqref{eq:Hessian.F.maximum}, at $(x_0,t_0)$ we have
$$
\mathcal{A}_9\le(m-\alpha)\theta |\nabla w|^2 \big((\theta')^{p-2}|\nabla w|^{p-2} + \varepsilon \big) \Big(\zeta^2\Big(\frac{\theta''}{\theta'} \Big)' |\nabla w|^2  - 2\zeta \frac{\theta''}{\theta'} \nabla \zeta(\nabla w)^T +3|\nabla \zeta|^2  - \Delta(\zeta^2/2)\Big).
$$

We want to get rid of the terms involving $|\nabla w|^4$ appearing in $\mathcal{A}_1$ and $\mathcal{A}_9$, since they would be leading order terms if $p\in[1,2]$. The concavity of $\theta$ does the job for $\mathcal{A}_1$, and we get
$$
\begin{aligned}
\mathcal{A}_1\le\,&(m-\alpha) (p-1)(\theta')^{p-2}\theta'' \zeta^2|\nabla w|^{p+2}-(m-\alpha)\big(  (p-1)(\theta')^{p-2}|\nabla w|^{p-2} + \varepsilon \big)\theta'\zeta|\nabla w|^2\nabla \zeta(\nabla w)^T.
\end{aligned}
$$

As for $\mathcal{A}_9$, if we require $(\theta''/\theta')'\le0$, then
$$
\begin{aligned}
\mathcal{A}_9\le&\,(m-\alpha)\theta(\theta')^{p-2}\Big(\frac{\theta''}{\theta'} \Big)' \zeta^2|\nabla w|^{p+2}\\
&+(m-\alpha)\theta |\nabla w|^2 \big((\theta')^{p-2}|\nabla w|^{p-2} + \varepsilon \big) \Big(- 2\zeta \frac{\theta''}{\theta'} \nabla \zeta(\nabla w)^T +3|\nabla \zeta|^2  - \Delta(\zeta^2/2)\Big).
\end{aligned}
$$

Starting from~\eqref{eq:grad_press_theta} and using the above estimates and Cauchy-Schwarz inequality we get, grouping the terms involving $|\nabla w|^{p+2}$ on the left hand side,
\begin{equation*}\label{eq:condition_theta}
\begin{aligned}
-(p-1)&(\theta')^{p-2}\mathcal{B}(\theta)\zeta^2|\nabla w|^{p+2}\\
\le\,&  c_1(\theta')^{p-3}\big((\theta')^{2}+\theta|\theta''|\big)\zeta\|\nabla \zeta\|_\infty|\nabla w|^{p+1} + c_2\theta(\theta')^{p-2}(\|\nabla\zeta\|_\infty^2+\|\operatorname{Hess}(\zeta)\|_\infty)|\nabla w|^{p}\\
&+ \varepsilon c_3\Big(\theta'+\frac{\theta|\theta''|}{\theta'}\Big)\zeta\|\nabla \zeta\|_\infty|\nabla w|^{3} + \varepsilon c_4\theta(\|\nabla\zeta\|_\infty^2+\|\operatorname{Hess}(\zeta)\|_\infty)|\nabla w|^{2}
\\
&+ \Big(\zeta\|\zeta_t\|_{L^\infty (Q_T)}+\zeta^2 \Big|\frac{\partial}{\partial w} \Big(\frac{h(\Phi(\theta))}{\Phi^\prime(\theta) \theta^\prime} \Big)\Big|\Big)|\nabla w|^{2},
\end{aligned}
\end{equation*}
where $c_i$, $i=1,\dots,4$, are positive constants, that do not depend at all on the solution, and
$$
\mathcal{B}(\theta):=(m-\alpha+1)\theta'' +  (m-\alpha)(p-2)\theta \left(\frac{\theta''}{\theta'}\right)^2 + (m-\alpha)\theta \left(\frac{\theta''}{\theta'}\right)'.
$$

We take $\theta(w):=aw(2-w)$  with $a=S/(2\delta - \delta^2)$, $\delta\in(0,1)$, so that  $\theta([0,\delta])=[0,S]$, and $\theta'>0$, $\theta''<0$,  $(\theta''/\theta')'\le0$ in $[0,\delta]$. On the other hand, since $a>0$ and $\alpha\in (0,m)$, if $\delta\in(0,1/2)$ we have
$$
\begin{aligned}
\mathcal{B}(\theta)&=-a\left(m-\alpha+1-(m-\alpha)(p-3)\frac{w(2-w)}{(1-w)^2} \right)\\
&\le \begin{cases}
-a(m-\alpha+1)\quad&\text{if } p\in(1,3],\\
-a(m-\alpha+1-(m-\alpha)(p-3)8\delta)\quad&\text{if } p>3.
\end{cases}
\end{aligned}
$$
Hence, taking $\delta=\delta(m,p)$ small enough we have $\mathcal{B}(\theta)\le-a(m-\alpha+1)/2<0$.

Once $\theta$ has been chosen, we proceed to estimate the term involving the reaction nonlinearity $h$. Since $\Phi^\prime=C_1\Phi^{1+\alpha-m}$, $\Phi^{\prime\prime} =C_2\Phi^{1+2\alpha-2m}$ for some constants $C_1\in\mathbb{R}_+$,  $C_2 \in\mathbb{R}$, and $\theta''=-2a$, then
$$
\begin{aligned}
\frac{\partial}{\partial w} \left(\frac{h(\Phi(\theta))}{\Phi^\prime(\theta) \theta^\prime} \right)
&= h^\prime(\Phi(\theta)) - \frac{h(\Phi(\theta))}{C_1^2 \Phi(\theta)}\Big( C_2 -\frac{2aC_1 \Phi^{m-\alpha}(\theta)}{(\theta')^2}\Big).
\end{aligned}
$$
Note that   $\theta'(w)=2a(1-w)\ge 2a(1-\delta)>a$ if $\delta\in (0,1/2)$. We also recall that $0\le \Phi^{m-\alpha}(\theta)/a\le CS/a\le C$. Hence, since
$|h(\Phi)/\Phi|\le \sup_{[0,M]}|h'|=:\mathcal{H}$,
we get
$$
\Big|\frac{\partial}{\partial w} \left(\frac{h(\Phi(\theta))}{\Phi^\prime(\theta) \theta^\prime} \right)\Big|\leq  C\mathcal{H}.
$$

We have finally arrived to
\begin{equation}
\label{eq:ineq.bernstein}
\begin{aligned}
\zeta^2|\nabla w|^{p+2}\leq\,&  d_1\zeta\|\nabla \zeta\|_\infty|\nabla w|^{p+1} + d_2(\|\nabla\zeta\|_\infty^2+\|\operatorname{Hess}(\zeta)\|_\infty)|\nabla w|^{p}\\
&+ \varepsilon d_3 S^{2-p}\zeta\|\nabla \zeta\|_\infty|\nabla w|^{3} + \varepsilon d_4S^{2-p}(\|\nabla\zeta\|_\infty^2+\|\operatorname{Hess}(\zeta)\|_\infty)|\nabla w|^{2}
\\
&+ d_5 \big(\zeta\|\zeta_t\|_{L^\infty (Q_T)}+\zeta^2\mathcal{H}\big)|\nabla w|^{2}\quad\text{at }(x_0,t_0),
\end{aligned}
\end{equation}
where $d_i$, $i=1,\dots,5$, are positive constants, that do not depend at all on the solution. Since $p\ge 2$, for all $\gamma>0$ there are constants $C_\gamma,D_\gamma>0$ such that
$$
\zeta|\nabla w|^{p+1}\le \gamma\zeta^2|\nabla w|^{p+2}+C_\gamma|\nabla w|^{p},\qquad
\zeta|\nabla w|^{3}\le \gamma\zeta^2|\nabla w|^{p+2}+D_\gamma|\nabla w|^{2},
$$
and hence (remember that $S\ge 1$)
$$
\begin{aligned}
\Big(1-\gamma &(d_1+\varepsilon d_3)\|\nabla \zeta\|_\infty\Big)\zeta^2|\nabla w|^{p+2}\leq \Big(d_2(\|\nabla\zeta\|_\infty^2+\|\operatorname{Hess}(\zeta)\|_\infty)+d_1 C_\gamma\|\nabla \zeta\|_\infty\Big)|\nabla w|^{p}\\
&+\Big(\varepsilon d_3D_\gamma\|\nabla \zeta\|_\infty + \varepsilon d_4\big(\|\nabla\zeta\|_\infty^2+\|\operatorname{Hess}(\zeta)\|_\infty\big)+ d_5 \big(\zeta\|\zeta_t\|_{L^\infty (Q_T)}+\zeta^2\mathcal{H}\big)\Big)|\nabla w|^{2}.
\end{aligned}
$$
From now on we  assume $\|\nabla\zeta\|_\infty\le 1$ and $\varepsilon\in (0,1)$, and take $\gamma=1/(2(d_1+d_3))$, so that
$$
1-\gamma (d_1+\varepsilon d_3)\|\nabla \zeta\|_\infty\ge\frac12.
$$

Since $p\ge 2$, if $|\nabla w|(x_0,t_0)\ge 1$, at $(x_0,t_0)$ we have
$\zeta^2|\nabla w|^{2}\leq \mathcal{A}$, where
$$
\mathcal{A}= C\big(\|\nabla\zeta\|_\infty^2+\|\operatorname{Hess}(\zeta)\|_\infty+\|\nabla \zeta\|_\infty+\|\zeta_t\|_{L^\infty (Q_T)}+\mathcal{H}\big),
$$
while if $|\nabla w|(x_0,t_0)\le 1$ then
$
\zeta^p|\nabla w|^{p}\leq \zeta^2|\nabla w|^{p} \le \mathcal{A}$.
We conclude that
\begin{equation}
\label{eq:estimate.z}
z\le \max_{Q_T} z\leq \max\{\mathcal{A},\mathcal{A}^{2/p}\}.
\end{equation}

Given $(x_1,t_1)\in Q$ we take
$$
\zeta_n(x,t)=\frac{t}{t_1} \psi\left(\frac{x-x_1}{n} \right),
$$
where $\psi$ is a compactly supported smooth function satisfying $0\leq \psi\leq 1$, $\psi = 1$ if $|x|\leq 1$  and $\psi = 0$ if $|x|\geq 2$, which is an admissible cut-off function in $Q_{t_1}$ for estimate~\eqref{eq:estimate.z} if $n$ is large enough, so that $\|\nabla\zeta_n\|_\infty\le 1$. With this choice, $\|\zeta_{n,t}\|_\infty=1/t_1$, $\|\nabla\zeta_n\|_\infty=O(1/n)$, $\|\operatorname{Hess}(\zeta_n)\|_\infty=O(1/n^2)$. Passing to the limit $n\to\infty$ we get
$$
|\nabla w|^2(x_1,t_1) = z(x_1,t_1) \leq \max\{C\big(t_1^{-1}+\mathcal{H}\big), \big(C(t_1^{-1}+\mathcal{H}\big)^{2/p}\},
$$
which means, since $x_1,t_1$ were arbitrary and $0\le\theta'\le 2S$, that
\begin{equation}
\label{eq:estimate.square.pressure}
|\nabla v|^2(x,t) \leq 4CS^2\max\{t^{-1}+\mathcal{H}, (t^{-1}+\mathcal{H})^{2/p}\}.
\end{equation}

If $p\in (0,2)$, we repeat the above computations with $z:=|\nabla w|^p\zeta^2$, and we get again~\eqref{eq:ineq.bernstein}. In this case $(x_0,t_0)$ is the point where the new $z$ achieves its maximum in $Q_T$, and the constants are different from the ones for $p\ge 2$. Now, if $|\nabla w|(x_0,t_0)\ge 1$, at $(x_0,t_0)$ we have
$\zeta^2|\nabla w|^{p}\leq \mathcal{A}$, with $\mathcal{A}$ as above,
while if $|\nabla w|(x_0,t_0)\le 1$ then
$
\zeta^2|\nabla w|^{p}\leq \zeta^p|\nabla w|^{p} \le \mathcal{A}^{p/2}$. After choosing the cut-off function and passing to the limit $n\to\infty$, we arrive to
$$
|\nabla w|^p(x,t)  \leq \max\{C\big(t^{-1}+\mathcal{H}\big), \big(C(t_1^{-1}+\mathcal{H})\big)^{p/2}\},
$$
and from here to~\eqref{eq:estimate.square.pressure}.

We end by taking the limit $\varepsilon\to 0$.
\end{proof}

Since $\nabla u^m=u^\alpha\nabla v$, we get the following corollary, which immediately yields Theorem~\ref{thm:bound.flux}.
\begin{Corollary}
	\label{cor:bound.flux}
	Let $u$ be a bounded nonnegative solution of problem~\eqref{eq:main}. With the notations of Lemma~\ref{lema_bounded_derivative_pressure}, there exists a constant $C$ independent of $u$ such that
	$$
	|\nabla u^m(x,t)| \leq CM^m\max\{(t^{-1}+\mathcal{H})^{1/2}, (t^{-1}+\mathcal{H})^{1/p}\}.
	$$
\end{Corollary}

\noindent\emph{Remark. }  It can be easily checked that the above proofs apply to bounded nonnegative solutions to equation~\eqref{eq:convection}. Indeed, the new term does not offer difficulties, since it will produce terms of (lower) order $p$ and $p+1$ in the inequality~\eqref{eq:ineq.bernstein}. As a corollary we conclude that the conclusion of Theorem~\ref{thm:bound.flux} is valid for such solutions, in fact with a constant $k$ uniformly valid for all $\gamma\in [0,\gamma_0]$ if $u_0$ is kept fixed. This will be useful later.


\section{ODE analysis}
\label{sect-subsupersolution} \setcounter{equation}{0}

In this section we gather several results that involve the analysis of ODEs. In the first subsection we construct specific sub- and supersolutions for the problem without convection to deal with one-dimensional situations. We devote the second subsection to construct wavefronts for the problem with convection~\eqref{eq:constant.convection}. Finally, in the third subsection we analyze the dependence of these wavefronts and their velocity on the convection parameter.

\subsection{One-dimensional sub- and supersolutions}
\label{subsect-one.dimensional}
We start by constructing sub- and supersolutions to~\eqref{eq:main} for $N=1$ of the form \begin{equation}\label{eq:def_sub_super}
w(x,t)=f(t)U_{c^*}(x-g(t)),
\end{equation}
following ideas from~\cite{Biro-2002,Garriz-2020}. The functions $f$ and $g$ will be required to solve a system of ODEs, so that $w$ approaches a travelling wave solution for large times, both in shape and speed.

When looking for a subsolution we will ask $f$ to grow to 1 very slowly so that it does not exceed our solution $u$ in size, while we will need a function $g$ that does not increase too fast for small times,  so that $w$ does not surpass $u$ through the boundary while the solution is \lq\lq starting to travel". For big times though, we need $g^\prime (t)\to c^*$.

When thinking about supersolutions we can allow $g$ to grow fast at the start, while maintaining the required behaviour in the limit. This speed also gives a bit more freedom when defining $f$, as we will see, but again it is better to make it go slowly to 1 to ensure that $w$ is above~$u$.

The ODE system for  $f$ and $g$  is given by
\begin{equation}\label{system:sub_super}
\begin{cases}
f^\prime (t) =  \varphi(f),\quad &f(0)=f_0 \in (1-\delta, 1+\delta), \\
g^\prime (t) = c^*f^{(p-1)m-1} -   k\varphi(f)/f,\quad &g(0)=g_0,
\end{cases}
\end{equation}
where $k>0$ is a big enough constant and $\delta$ is such that $h^\prime(u)<0$ for all $u\in(1-\delta,1+\delta)$. This interval must exist thanks to the continuity of $h^\prime$ at  $u=1$ and condition~\eqref{eq:reaction_deriv_in_1}.
The function $\varphi:[1-\delta,1+\delta] \rightarrow \mathbb{R}$ is taken such that:
\begin{itemize}
\item $\varphi(1)=0$, $\varphi^\prime(1)<0$, $\varphi>0$ in $[1-\delta,1)$ and $\varphi<0$ in $(1,1+\delta]$.

\item $\ \varphi$ is continuously differentiable (hence  Lipschitz continuous) in its domain.

\item $\sup\limits_{[1-\delta,1+\delta]} |\varphi^\prime| \leq H(1-\delta)^{(p-1)m}$, where $H=\inf\limits_{[1-\delta,1+\delta]} |h^\prime|$.
\end{itemize}

The proof of the next lemma is  rather straightforward; hence we omit it. Versions for the porous medium case have already been proved in~\cite{Biro-2002,Garriz-2020}.
\begin{Lemma}\label{lema1} Let~\eqref{eq:slow.regime} hold. Solutions $(f,g)$ to the ODE system~\eqref{system:sub_super} satisfy the following properties:
\begin{itemize}
\item[\rm (a)] $\lim\limits_{t \rightarrow \infty} f(t) = 1$.

\item[\rm (b)] If $1-\delta < f_0 < 1$ then $f$ is strictly monotone increasing, and hence $g(t) - c^*t$ is strictly monotone decreasing.

\item[\rm (c)]If $1+\delta>f_0 >1$ then $f$ is strictly monotone decreasing, and hence $g(t) - c^*t$ is strictly monotone increasing.

\item[\rm (d)] $\lim\limits_{t \rightarrow \infty} g(t) = \infty$.


\item[\rm (e)] $\lim\limits_{t \rightarrow \infty} g^\prime(t) = c^*$.

\item[\rm (f)] There exists $\xi_0 \in \mathbb{R}$ such that $\lim\limits_{t \rightarrow \infty} (g(t) - c^* t)= \xi_0$.
\end{itemize}
\end{Lemma}

With the above choices of $f$ and $g$ the function $w$ defined by~\eqref{subsect-one.dimensional} is indeed a sub- or a supersolution; see~\cite{Garriz-2020} for a proof for the porous medium case, which can be easily adapted to the general case.

\begin{Lemma}\label{lema2}
Let~\eqref{eq:slow.regime} hold and $N=1$.
Let $w$ beas in~\eqref{eq:def_sub_super} and $f$ and $g$ be as in~\eqref{system:sub_super}. Then:
\begin{itemize}
\item[(i)] If  $f_0\in(1-\delta, 1)$, then $w$ is a subsolution of~\eqref{eq:main}.
\item[(ii)] If  $f_0\in (1,1+\delta)$, then $w$ is a supersolution of~\eqref{eq:main}.
\end{itemize}
In both cases there exists $\xi_0 \in \mathbb{R}$ such that
$$
\lim\limits_{t \rightarrow \infty} w(\xi + c^*t, t) = U_{c^*}(\xi - \xi_0)
\quad\text{uniformly with respect to }\xi \in \mathbb{R}.
$$
\end{Lemma}

\subsection{Wavefronts for the problem with convection}
\label{subsect-wavefronts.convection}

We now look for travelling wave solutions to~\eqref{eq:convection}, that is, solutions of the form $u(r,t) = U(\xi)$, $\xi =r -ct $, with a nonincreasing profile $U$ connecting 1 to 0.   If we define
$\phi(\xi)=|(U^m)^\prime (\xi)|^{p-1}$, which  corresponds to the absolute value of the flux, our equation transforms into
\begin{equation*}\label{tweq}
-\phi' -\gamma \phi  + cU^\prime + h(U) =0,
\end{equation*}
and we get the system
\begin{equation}\label{eq:original_system}
U^\prime = -\frac{\phi ^\alpha}{mU^{m-1}}, \qquad \phi ^\prime = h(U)  -\gamma \phi -\frac{c\phi ^\alpha}{mU^{m-1}}.
\end{equation}
Defining $f(U) = mU^{m-1}h(U)$ and $\textrm{d}\xi = mU^{m-1} \, \textrm{d}\tau\ $ we arrive to the \lq\lq less singular'' system
\begin{equation}\label{eq:linearizable_system}
\dot U = -\phi ^\alpha, \qquad \dot \phi = f(U)  -\gamma m \phi  U^{m-1} -c\phi ^\alpha,
\end{equation}
where $\dot{}=\textrm{d}/\textrm{d}\tau$, with an equation for the trajectories $\phi (U)$ given by
\begin{equation}\label{eq:trajectories}
\frac{\textrm{d}\phi }{\textrm{d}U} = c+\gamma m U^{m-1}\phi ^{1-\alpha} - \frac{f(U)}{\phi ^{\alpha}}.
\end{equation}

System~\eqref{eq:linearizable_system} is very similar to the one  studied by Aronson and Weinberger in~\cite[Chapter 4]{Aronson-Weinberger-1978}, if we make the correspondence $(U,\phi)\leftrightarrow (q,p)$. In fact, we recover theirs by taking $m=1$, $\alpha=1$, $\gamma=0$, as expected.  Since our proofs will work out in a similar way,   we will only say something about them or the intermediate results when  there is a significant change with respect to the ones in~\cite{Aronson-Weinberger-1978}.

One can integrate by separation of variables in the equation $U^\prime = -\phi ^{\alpha}/mU^{m-1}$, before the change of variables $\xi \leftrightarrow \tau$, to get
\begin{equation}\label{eq:support_tw}
\xi_1-\xi_0 = m \int_{U(\xi_1)}^{U(\xi_0)} \frac{U^{m-1}}{(\phi (U))^\alpha}\,\textrm{d}U.
\end{equation}
Therefore, since $\alpha<m$, if a trajectory connecting $(1,0)$ with $(0,0)$   enters the origin with a finite slope different from 0, it will correspond to a finite wavefront. We will show that for all $\gamma\ge0$ small there is one, and only one,  positive value $c(\gamma)$ for which such a connection exists.
There may be other connecting trajectories for $c>c(\gamma)$, but they will enter the origin with slope~0, and they are not finite.

Let us begin the study of the system and the trajectories. Given $c \geq 0$ and $\nu>0$, let $\phi _c(U;\nu)$ be the only trajectory that passes through the point $(0,\nu)$, and
$$
U_{c,\nu}:=\sup\{\eta\in (0,1] :  \phi _c(U; \nu) >0\text{ for }U\in [0,\eta)\}.
$$
For convenience, if $U_{c,\nu} < 1$ we redefine $\phi _c(U;\nu)=0$ for $U\in [U_{c,\nu},1]$. Since $0<\nu<\mu$ implies $0\leq \phi _c(U;\nu) \leq \phi _c(U;\mu)$ for each $U \in [0,1]$, we can define
$$
\phi _c(U) = \lim\limits_{\nu \rightarrow 0} \phi _c(U;\nu),\quad U\in[0,1].
$$
Suppose that there exists a value $U^{(c)} \in (0,1]$ such that $\phi _c(U)>0$ in $(0,U^{(c)})$ and $\phi _c(U^{(c)})=0$ in case $U^{(c)} \neq 1$. It follows from the Monotone Convergence Theorem that $\phi_c(U)$ is a solution of \eqref{eq:linearizable_system} in $(0,U^{(c)})$. If $U^{(c)} <1$, we redefine $\phi_c(U) = 0$ for $U\in [U^{(c)},1]$ if necessary.
Let $S = \{ (U,\phi ): 0<U<1, \phi >0 \}$ and
$T_c=S\cap \{ (U,\phi ): 0<U<1, \phi = \phi _c(U)\}$.
Note that $T_c=\emptyset$ if $\phi_c\equiv0$. If it is non-empty, it is a curve through $(0,0)$.

To study the behaviour of the trajectories $T_c$, which will determine the critical value for $c$, we need the following comparison lemma, which is nothing but~\cite[Lemma 4.1]{Aronson-Weinberger-1978}: Let $\phi_1$, $\phi_2$ be functions defined in $[a,b]$ satisfying
$$
\frac{\textrm{d} \phi _1}{\textrm{d} U}=F_1(U,\phi _1)\quad\text{and} \quad\frac{\textrm{d} \phi _2}{\textrm{d} U}=F_2(U, \phi _2)\quad\text{in }(a,b).
$$
If $\phi _1(a)>\phi _2(a)$,
and either
$F_1(U, \phi _1)>F_2(U,\phi _1)$ or $F_1(U, \phi _2)>F_2(U,\phi _2)$ for all $U\in(a,b)$,
then $\phi _1(U)\geq \phi _2(U)$ for all $U\in[a,b]$.

 The next lemma corresponds to~\cite[Propositions 4.1--4.2]{Aronson-Weinberger-1978}.

\begin{Lemma}\label{lema_pp_1}
\emph{(a)} There are constants $k_1,k_2>0$ such that
\begin{equation}\label{eq:bound.above.phi_c}
\phi _c(U) < k_1 + k_2U\quad\text{for } U\in[0,1].
\end{equation}

\noindent \emph{(b)}	Let  $\sigma_0:=\sup\limits_{U\in [0,1]} h(U)/U$. If
	\begin{equation}
	\label{eq:cond.below.c}
	c> (\alpha + 1)\left(\frac{m\sigma_0}{\alpha^\alpha}\right)^{\frac{1}{\alpha+1}}\quad\text{and }\gamma\ge0,
	\end{equation}
then
	$$
	\displaystyle\frac{\alpha c}{\alpha+1} U \leq \phi _c(U) < k_1 + k_2U\quad\text{for } U\in[0,1].
	$$
In particular $U^{(c)}=1$, and hence	the curve $T_c$ is a trajectory of system~\eqref{eq:linearizable_system} in $S$ through $(0,0)$. It is maximal, in the sense that no other trajectory through $(0,0)$ has points in $S$ above $T_c$.
\end{Lemma}
\begin{proof}
(a) We consider $\phi _1(U)=k_1 + k_2 U$  and $\phi _2(U) = \phi _c(U;\nu)$ for certain  $k_1,k_2>0$ to be chosen. The inequality $F_1(U;\phi _1)>F_2(U;\phi _1)$ in [0,1] is equivalent to
	$$
	c+\gamma (k_1 + k_2)^{1-\alpha} + \frac{m  |\min h(U)|}{k_1^\alpha} \leq k_2,
	$$
	and such a couple of values $k_1$ and $k_2$ must exist, the exact threshold does not matter. We finish by taking $\nu<k_1$ and letting $\nu\to 0$.
	
\noindent (b) Now we take  $\phi _1(U)=\phi _c(U;\nu)$  and $\phi _2(U) = zU$ for a certain positive value of $z$ to be determined. In order to have $F_1(U;\phi _2)>F_2(U;\phi _2)$  we need
\begin{equation*}
\label{eq:condition.comparison}
c+\left(\gamma m z^{1-\alpha} - \frac{mh(U)}{z^\alpha U}\right)U^{m-\alpha} > z.
\end{equation*}
Since the left-hand side is greater or equal than $c-\frac{m\sigma_0}{z^\alpha}$ for all $U\in[0,1]$, 	it will be enough to get a value $z>0$ such that  $P(z):=z^{\alpha+1} - cz^\alpha + m\sigma_0< 0$. This can be fulfilled if~\eqref{eq:cond.below.c} holds,  since this is precisely the required condition for  the \lq\lq fractional polynomial''  $P$  to reach negative values in $\mathbb{R}_+$. ~If~\eqref{eq:cond.below.c} holds, the value $z=c/(\alpha+1)$ will do the job. Therefore $\phi _1(U)\geq \phi _2(U)$ in~$[0,1]$. We finish by letting $\nu \to 0$. Maximality comes from the way we have constructed $\phi_c$.
	\end{proof}

\noindent\emph{Remark. } Given $c_0,\gamma_0>0$, we can choose constants $k_1,k_2>0$ in~\eqref{eq:bound.above.phi_c} valid for all $c\le c_0$, $\gamma\le\gamma_0$. Hence, since $\phi_c$ is maximal, the line $\phi (U)=k_1 + k_2 U$ represents a universal barrier for \emph{all} trajectories when the parameters are in this range, which yields a bound for the values of $\phi$.

\medskip

Condition~\eqref{eq:cond.below.c} provides an upper bound for the critical speed $c(\gamma)$, that we define later, independent of the convection parameter~$\gamma$. If we are a bit more careful in the comparison and do not disregard the effect of convection we can obtain a better bound, as we shall see next.

\begin{Lemma}\label{lema_pp_2}
	Suppose that
	\begin{equation}
	\label{eq:condition.c.convection}
	\left(\frac{c}{\alpha+1}\right)^{\alpha+1}\alpha^{\alpha} +  \left(\frac{\gamma m}{\alpha+1}\right)^{\frac{1}{\alpha}+1}\alpha > m\sigma_0,
	\end{equation}
	where $\sigma_0$ is as in the previous lemma. Then $U^{(c)} =1$ and $\phi_c(1) > 0$.
\end{Lemma}
\begin{proof}
	
	Let us take $\phi _2(U)=zU$ for a certain positive constant $z$ and $\phi _1(U)=\phi _c(U;\nu)$ for a certain positive $\nu$. As in the previous lemma,
	$$
	z^{\alpha+1} - cz^\alpha- m\left(\gamma  z - \sigma_0\right)U^{m-\alpha}< 0
	$$
guarantees that we can apply comparison.
In the worst scenario, where $\gamma z-\sigma_0<0$, the left-hand quantity in the inequality is smaller than $P(z):=z^{\alpha +1} - \gamma m z - c z^\alpha + m\sigma_0$ for all $U\in[0,1]$. Hence, it is enough to have $P(z)<0$ for some positive $z>0$. This is  certainly possible if  if~\eqref{eq:condition.c.convection} holds, since then $P(z_0)< 0$ for
	$$
	z_0:= \left(\left(\frac{c\alpha}{\alpha+1}\right)^{\alpha+1} + \left(\frac{\gamma m}{\alpha+1}\right)^{\frac{1}{\alpha}+1} \right)^{\frac{1}{\alpha+1}}.
	$$
We finish by taking $\nu\to 0$.
\end{proof}

Our next result shows that if $\gamma$ is small then $c(\gamma)$ has to be positive.
\begin{Lemma}\label{lema_pp_3}
	There exists a constant $\gamma^*>0$ such that if $\max\limits_{U\in [0,1]} \displaystyle\int_0^U f(u)\,{\rm d}u >0$, $\gamma< \gamma^*$ and $c\leq 0$ then no trajectory connects the points $(0,0)$ and $(0,1)$.
\end{Lemma}
The proof  is similar to the one  of~\cite[Proposition 4.3]{Aronson-Weinberger-1978}, once one recalls that any trajectory $\phi (U)$ is uniformly bounded for $\gamma\in[0,\gamma_0]$ and $c\le0$; see the remark after Lemma~\ref{lema_pp_1}.  This allows to make the term $\gamma m \int_0^U s^{m-1} \phi (s)\,\textrm{d}s$,  appearing when we integrate the equation of the trajectories~\eqref{eq:trajectories} multiplied by $\phi^\alpha$, arbitrarily small by taking $\gamma$ small, which is enough to proceed.

In view of these results, the \emph{critical speed}
$c(\gamma):=\inf\{c>0: U^{(c)}=1, \phi _c(1) >0 \}$
is well defined whenever $\max\limits_{U\in [0,1]} \displaystyle\int_0^U f(u)\,{\rm d}u >0$ and $\gamma\in [0,\gamma^*)$, and satisfies
$$
0\leq c(\gamma) \leq \left(m\sigma_0 - \alpha \left(\frac{\gamma m}{\alpha+1}\right)^{\frac{1}{\alpha} + 1} \right)^{\frac{1}{\alpha+1}} \frac{\alpha+1 } {\alpha^{\frac{\alpha}{\alpha+1}}}.
$$

\noindent\emph{Remark. } Proceeding as in the proof of~\cite[Proposition 4.3]{Aronson-Weinberger-1978}, it is easy to check that $c(0)>0$ also when $p\neq2$.

\medskip

The proof of the  next  lemma is similar to that of the corresponding ones in~\cite[Lemma 4.2 and Proposition 4.5]{Aronson-Weinberger-1978}.

\begin{Lemma}\label{lema_pp_4}
	If $0 \leq c < d$ and $T_c$ is not empty, then $\phi _d(U) < \phi _c(U)$ in $(0,U^{(c)}]$. Moreover,  $\lim\limits_{c \rightarrow d} \phi _c(U)=\phi _d(U)$. Thus, the family $\{T_c\}$ is continuous in the parameter $c$ whenever $c>0$.
\end{Lemma}

Taking into account all the above results we obtain, following the proof of~\cite[Theorem 4.1]{Aronson-Weinberger-1978}, the existence of the desired travelling wave for all small nonnegative values of $\gamma$.

\begin{Theorem}
\label{thm:existe.tw.convection}
Assume~\eqref{eq:slow.regime}  and $\gamma\in[0,\gamma_*)$, with $\gamma_*$ as in Lemma~\ref{lema_pp_3}.
	If $c=c(\gamma)$, there exists a unique trajectory connecting the points~$(0,0)$ and $(1,0)$ in system~\eqref{eq:linearizable_system}.
This trajectory yields a finite wavefront for the equation in~\eqref{eq:main}. Such a wavefront is unique up to translations.
\end{Theorem}
Let us recall that
\begin{equation}\label{eq:behaviour.origin}
\phi _{c(\gamma)}(U)=c(\gamma) U(1+o(1))\quad\text{when }U\sim 0.
\end{equation}
As for the behaviour for $U\sim 1$ we have the following result, which has already been obtained for the problem without convection and certain reaction nonlinearities in~\cite{Audrito-2019,Audrito-Vazquez-2017}.
\begin{Lemma}\label{lemma:behaviour_near_1}
Under the assumptions of Theorem~\ref{thm:existe.tw.convection},  the maximal trajectory connecting $(0,0)$ with $(1,0)$ enters the latter point with behaviour
\begin{equation}\label{eq:behaviour.U.1}
\begin{aligned}
&\phi _{c(\gamma)}(U)=C_{p,\gamma}(1-U)^{\mu_{p,\gamma}}(1+o(1))\quad\text{when }U\sim 1,\quad\text{where}\\
&C_{p,\gamma}=
\begin{cases}
\left(\frac{m|h'(1)|}{c(\gamma)}\right)^{p-1}& \mbox{if } p>2, \\
\frac{-(c(\gamma)+\gamma m)+ \sqrt{(c(\gamma)+\gamma m)^2+4m|h'(1)|}}{2}\quad  & \mbox{if } p=2, \\
\frac{|h'(1)|}{\gamma} & \mbox{if }p\in(1,2), \ \gamma>0,
\end{cases}
\\
&\mu_{p,\gamma}=\begin{cases}
p-1&\text{if }p\ge2,\\
1&\text{if } p\in(1,2),\ \gamma>0.
\end{cases}
\end{aligned}
\end{equation}
\end{Lemma}
\begin{proof}
When $p\in(1,2)$,  which means that $\alpha>1$, system~\eqref{eq:linearizable_system} is linearizable for $(U,\phi)\sim(1,0)$. There are two real eigenvalues, $\lambda_-=-m\gamma$ and $\lambda_0=0$, with eigenvectors $v_-=(1,0)$ and $v_0=(1,-|h'(1)|/\gamma)$. Hence we have a trajectory entering $(1,0)$ along the direction of $v_-$, and one or more trajectories exiting $(1,0)$ into $S$ along $v_0$. To show that there is only one exiting orbit behaving like $\phi(U)\sim (|h'(1)|/\gamma)(1-U)$ close to $(1,0)$, we will prove that any such trajectory is  repulsive near~$(1,0)$. Since trajectories satisfy $\phi'(U)=F(U,\phi(U))$,
with
$$
F(U,\phi)= c+\gamma m U^{m-1}\phi ^{1-\alpha} - \frac{f(U)}{\phi ^{\alpha}},
$$
see~\eqref{eq:trajectories}, it is enough to show  that $\partial_\phi F(U,\phi)>0$ along any trajectory of this kind  near $U=1$. Indeed, if $\phi(U)\sim (|h'(1)|/\gamma)(1-U)$, then
$$
\partial_\phi F(U,\phi)=\frac{(1- \alpha)\gamma m U^{m-1}}{\phi ^{\alpha}} + \frac{\alpha f(U)}{\phi ^{\alpha+1}}\sim \frac{\gamma^{\alpha+1} m}{(1-U)^\alpha|h'(1)|^\alpha} \gg 0\quad  \text{for }U\sim 1.
$$

When  $p\geq 2$, following what was done in~\cite[Section 3]{Audrito-Vazquez-2017} for the case without convection, we rewrite the system in terms of the  variables $U$ (the density) and $Z:=\left(\frac{m}{m-\alpha}U^{m-\alpha}\right)'$ (the derivative of the pressure). The new system,
$$
U'=- (p-1)|Z|^{p-1}U,\quad Z'=cZ - |Z|^p-\gamma U |Z|^{p-1}+mU^{m-\alpha-1}h(U),
$$
has the advantage of being linearizable for $(U,Z)\sim (1,0)$. We can now proceed as in the case $p\in(1,2)$, and rewriting the result for the trayectory $Z(U)$ in terms of the variables $U$ and $\phi$ we get the result.
\end{proof}

\noindent\emph{Remark. } (a)
If $p\in(1,2)$,  the slope of the connecting trajectory at $U=1$ goes to $-\infty$ as $\gamma\to 0^+$, resembling the vertical nature of the trajectory in the case $\gamma=0$, which has a behaviour
\begin{equation}\label{eq:gamma.0}
\phi_{c(0)}(U)=\Big(\frac{pm|h'(1)|}{2(p-1)}\Big)^{(p-1)/p}(1-U)^{\frac{2(p-1)}{p}}\quad\text{if }p\in(1,2),\ \gamma=0;
\end{equation}
see~\cite{Audrito-Vazquez-2017}.

\noindent (b) The same behaviours hold true for the only trajectory exiting $(0,1)$ and entering $S$ for any $c>0$ (not necessarily $c(\gamma)$) and $\gamma\ge0$.

\medskip

We now translate the above behaviours for the maximal connecting trajectory to the finite wavefront $U(\xi;\gamma)$ with support $\overline{\mathbb{R}_-}$ and the corresponding pressure
$$
V(\xi;\gamma):=\frac {m}{m-\alpha}(U(\xi;\gamma))^{m-\alpha}.
$$
We start with the behaviour at the origin, where $U$ approaches $0$.
\begin{Lemma}
\label{lema_limite_segunda_derivada}
Assume~\eqref{eq:slow.regime}. Then $\lim\limits_{\xi\to0^-}V'(\xi;\gamma)=- (c(\gamma))^\alpha$, and
$$
\displaystyle\lim\limits_{\xi \to0^-} V''(\xi;\gamma) = \frac{(\gamma c(\gamma) - h'(0))(m-\alpha)}{(p-1)(m-\alpha+1)(c(\gamma))^{1-\alpha}},\quad
\displaystyle\lim\limits_{\xi\to0^-}\Delta_p V(\xi;\gamma) = \frac{(\gamma c(\gamma) - h'(0))(m-\alpha)}{m-\alpha+1}.
$$
\end{Lemma}
\begin{proof}
The value of the derivative of the pressure at the origin follows from~\eqref{eq:behaviour.origin},
\begin{equation*}\label{eq:frontera}
\lim\limits_{\xi\to0^-}V'(\xi;\gamma)=\lim\limits_{\xi\to 0^+}  m U^{m-1-\alpha}(\xi;\gamma)U'(\xi;\gamma)=-\lim\limits_{\xi\to 0^-} \left(\frac{\phi _{c(\gamma)}(\xi)}{U(\xi;\gamma)}\right)^\alpha= -(c(\gamma))^\alpha.
\end{equation*}
Starting from~\eqref{eq:original_system} it is easy to check that
$$
V''= -\alpha \left(\frac{\phi }{U}\right)^{\alpha-1}\left[ \frac{h(U)}{U} - \gamma\frac{\phi }{U} +\frac{1}{m}\left(\frac{\phi }{U}\right)^{\alpha} \left(\frac{\big(\phi_c/U\big) - c}{U^{m-\alpha}}\right) \right].
$$
Therefore
$$
\lim\limits_{\xi\to0^-} V ''(\xi;\gamma) =-\alpha(c(\gamma))^{\alpha-1}\left[ h'(0)-\gamma c(\gamma)+\frac{(c(\gamma))^\alpha}{m}  \lim\limits_{\xi\to0^-}\left(\frac{\big(\phi_{c(\gamma)}(\xi)/U(\xi;\gamma)\big)   - c}{U^{m-\alpha}(\xi;\gamma)}\right)\right],
$$
if the limit appearing at the end of the right-hand side exists. Assume that it exists and call it $\beta$. Using l'H\^opital's rule and  the relation $\phi_c/U=(-V')^{1/\alpha}$, that comes from~\eqref{eq:original_system} and the definition of the pressure, we obtain
$$
\beta=\frac{m}{m-\alpha +1}\big(\gamma c(\gamma)-h'(0)\big) (c(\gamma))^{-\alpha},
$$
from where the limit for $V''$ follows.  To check that the limit $\beta$ indeed exists, we can follow the lines of the proof of~\cite[Lemma 2.2]{Du-Quiros-Zhou-Preprint}.

The result for $\Delta_p V$ follows immediately from $\Delta_p V=(p-1)|V'|^{p-2} V''$.
\end{proof}
We now give the behaviour at $-\infty$, where $U$ approaches 1.
\begin{Lemma}
\label{lem:behaviour.pressure.infinity}
Assume~\eqref{eq:slow.regime}.

\noindent{\rm (a)} $\lim\limits_{\xi\to -\infty}V'(\xi;\gamma)=0$.

\noindent{\rm (b)} Given $\gamma\in[0,\gamma_*)$ and $p>1$, there exist constants $\lambda, M>0$, $\xi_*<0$,  such that for all $\xi \in (-\infty, \xi_*]$,
$$
1>U(\xi;\gamma)\geq\begin{cases}
                     1 - M{\rm e}^{-\lambda|\xi|} & \mbox{if } p\ge2, \\
                     1-M\gamma^{\frac{p-1}{2-p}}|\xi|^{-\frac{p-1}{2-p}} & \mbox{if }p\in(1,2),\ \gamma>0.
                   \end{cases}
$$
\end{Lemma}
\begin{proof}
(a) The limit follows easily from~\eqref{eq:behaviour.U.1}, or \eqref{eq:gamma.0} when $p\in(1,2)$ and $\gamma=0$, together with
the relation $V'= -\big(\phi/U\big)^\alpha$.

\noindent (b) We do the computation for the case $p\ge2$, the other one being similar.   Let $\xi_*\in\mathbb{R}_-$ be such that $U(\xi_*)<1$, $U(\xi_*)\sim 1$, and let $\xi<\xi_*$. Behaviour~\eqref{eq:behaviour.U.1} implies that
there are constants $c,C>0$ such that
$$
c\le \frac{\int_{U(\xi_*)}^{U(\xi)} \frac{s^{m-1}}{(\phi (s))^\alpha}\,\textrm{d}s}
 {\int_{U(\xi_*)}^{U(\xi)}(1-s)^{-1}\,\textrm{d}s}\le C.
$$
Since $\int_{U(\xi_*)}^{U(\xi)}(1-s)^{-1}\,\textrm{d}s= \log\left(\frac{1-U(\xi_*)}{1-U(\xi)}\right)$,
formula~\eqref{eq:support_tw} with $\xi_1=\xi_*$ and $\xi_0=\xi$ implies
$$
(1-U(\xi_*)){\rm e}^{\xi/ C}{\rm e}^{-\xi_*/C}\ge 1-U(\xi)\ge (1-U(\xi_*)){\rm e}^{\xi/ c}{\rm e}^{-\xi_*/c}>0,
$$
from where the result follows with $\lambda=1/C$ and $M=(1-U(\xi_*)){\rm e}^{-\xi_*/C}$.
\end{proof}
\noindent\emph{Remark. } When $p\in(1,2)$ and $\gamma=0$ we have behaviour~\eqref{eq:gamma.0}, which translates to
$$
1>U(\xi)\ge  1-C|\xi|^{-\frac{p}{2-p}}\quad\text{for }\xi\in(-\infty,\xi_*].
$$
As in the case with convection, we  have a power-like approach to 1, though now it is  faster. In any case, both with and without convection, the singular nature of the diffusion for $p\in(1,2)$  dominates the exponential behaviour given by the linear slope of the reaction $h$ at $u=1$. An easy  comparison argument with travelling waves shows, in this range of values of $p$, that  spreading solutions to problem~\eqref{eq:main} with a compactly supported initial datum strictly below 1 approach that value at most in a power-like manner.

\medskip

We end this subsection with a result that will help us to build certain subsolutions which will serve as comparison functions in Section~\ref{Over the Fujita exponent}.

Let us define the quantity
$$
\beta_c  = \left\{ \begin{array}{lcc}
0 \text{ if there exists no curve in } S\text{ through } (0,0), \\
U^{(c)} \text{ if the extremal trajectory } T_c\text{ in } S \text{ through }(0,0)\text{ exists.}
\end{array}
\right.
$$
\begin{Lemma}\label{lema_pp_5}
Let $\gamma\in[0,\gamma_*)$. If $c \in (0,c(\gamma))$, then $\beta_c \in [0,1)$, and for every $\eta \in (\beta_c, 1)$ the only trajectory through $(\eta,0)$ enters $S$ through that point and leaves it through a point $(0,\nu)$, $\nu>0$, in the positive $\phi$-axis.
\end{Lemma}
The proof is analogous to that for the semilinear case given in~\cite[Lemma 4.3]{Aronson-Weinberger-1978}.

It is not difficult to check that the behaviour close to $(\eta,0)$ of the trajectory provided by Lemma~\ref{lema_pp_5} is given by $\phi(U)\sim C (\eta- U)^{1/(\alpha+1)}$. Therefore, the corresponding profile $U^{c,\eta}(\xi)$ reaches the height $\eta$, with 0 flux, at a finite value $\xi_0$; see relation~\eqref{eq:support_tw}. On the other hand, the same relation shows that $U^{c,\eta}(\xi)$ first touches 0, with a negative flux $-\nu$, at a finite value $\xi_1$. Let $b=\xi_1-\xi_0$. After a translation, we may set $\xi_0=0$, $\xi_1=b$. We summarize all this information in the following lemma.
\begin{Lemma}
\label{lem:profile.subsolution}
Let $\gamma\in[0,\gamma_*)$, $c \in (0,c(\gamma))$, and $\eta \in (\beta_c, 1)$. There is a unique monotone decreasing function $U^{c,\eta}$ such that,
for some $b,\nu>0$,
\begin{equation}\label{eq:profile.subsolution}
\begin{cases}
c(U^{c,\eta})'- \big(|((U^{c,\eta})^m)'|^{p-1}\big)'-\gamma|((U^{\eta,c})^m)'|^{p-1}+ h(U^{\eta,c})=0\quad&\text{in }(0,b),\\
U^{c,\eta}(0)=\eta,\quad (U^{c,\eta})^m)'(0)=0, \quad U^{c,\eta}(b)=0,\quad (U^{c,\eta})^m)'(0)=-\nu.
\end{cases}
  \end{equation}
\end{Lemma}
The profiles $U^{c,\eta}$,  extended adequately, will be used later to construct subsolutions to~\eqref{eq:main}.

\subsection{Dependence on the convection parameter}\label{subsect-dependence.convection}

The purpose of this subsection is to study the dependence on $\gamma$ of both $c(\gamma)$ and $U(\cdot;\gamma)$ in the interval $[0,\gamma_*)$, where we know that they are well defined. It is convenient to work with the pressure , which is already known to satisfy $V(\cdot;\gamma)\in C(\mathbb{R})\cap C^\infty(\overline{\mathbb{R}_-})$, $V(-\infty;\gamma)=m/(m-\alpha)$, and
$$
-c(\gamma)V' - (m-\alpha)V \Delta_p V - |V'|^p + \gamma (m-\alpha) |V'|^{p-1} V=m\Big(\frac{m-\alpha}{m}V\Big)^{1-\frac{1}{m-\alpha}}h\Big( \Big(\frac{m-\alpha}{m}V\Big)^{\frac{1}{m-\alpha}} \Big)
$$
in $\mathbb{R}_-$. We have omitted the arguments of the pressure  to simplify the writing.
If we differentiate this equation with respect to $\xi$, we get
$$
\begin{array}{l}
\quad-c(\gamma)V'' + ((m-\alpha)(p-1) + p)|V'|^{p-1}V'' - (m-\alpha)(p-1)V|V'|^{p-3}\big(|V'| V''' - (p-2)(V'')^2 \big)\\[8pt]
\quad-\gamma(m-\alpha)((p-1)V|V'|^{p-2}V'' + |V'|^p) \\[8pt]
= (m-\alpha-1)\big( \frac{m-\alpha}{m}V \big)^{-\frac{1}{m-\alpha}}V'h\left(\big( \frac{m-\alpha}{m}V \big)^{\frac{1}{m-\alpha}}\right)
+ V'h'\left(\big( \frac{m-\alpha}{m}V \big)^{\frac{1}{m-\alpha}}\right).
\end{array}
$$
The only qualitative difference between the general case that we are studying here and the porous medium case, where $\alpha=1$, resides in the powers in the reaction term. This difference adds few difficulties to reproduce the study in~\cite{Du-Quiros-Zhou-Preprint} that allows  to conclude that $c(\gamma)$ belongs to $C^2([0,\gamma_0))$ for some $\gamma_0>0$.  Hence we will omit most of the computations and show them only when they differ significantly from the ones in~\cite{Du-Quiros-Zhou-Preprint}.

 The next lemma states that the function $c(\gamma)$ is locally Lipschitz in  $[0,\gamma_*)$.
\begin{Lemma}\label{lema_c_lipschitz}
Given  any $\gamma_0\in(0,\gamma_*)$,  there is a positive constant $k$ depending on $\gamma_0$ such that
$$
0\leq c(\gamma_1) - c(\gamma_2)\leq k(\gamma_2 - \gamma_1)\quad\text{if } 0\le \gamma_1\le \gamma_2\le \gamma_0.
$$
\end{Lemma}
The proof is similar to the one for the case $p=2$ in~\cite[Lemma 2.3]{Du-Quiros-Zhou-Preprint}, which is an adaptation of~\cite[Proposition 3.1]{Malaguti-Ruggerini-2010},  recalling that  we can find uniform bounds for the gradient of the pressure and the flux  of the travelling wave whenever $\gamma$ is bounded; see the remarks after Corollary~\ref{lema_bounded_derivative_pressure} and  Lemma~\ref{lema_pp_1}.

To proceed we consider the derivative of $V(\cdot; \gamma)$ at level $q$, denoted by
$$
\mathcal{P}(q;\gamma):=V'(V^{-1}(q;\gamma); \gamma), \quad q\in\left(0, \frac{m}{m-\alpha} \right),
$$
which satisfies
$$
\big(|\mathcal{P}|^{p-1}\big)'=\frac{c(\gamma) - |\mathcal{P}|^{p-1}}{(m-\alpha)q} +\gamma |\mathcal{P}|^{p-2}-  \frac{1}{|\mathcal{P}|}\left(\frac{m-\alpha}{m}q\right)^{-\frac{1}{m-\alpha}} h\left(\left(\frac{m-\alpha}{m}q\right)^{\frac{1}{m-\alpha}}\right).
$$
 The lemma below follows easily from lemmas~\ref{lemma:behaviour_near_1}--\ref{lema_limite_segunda_derivada}.

\begin{Lemma}\label{lemma:limits_derivative}
Assume~\eqref{eq:slow.regime} and $\gamma\in [0,\gamma^*)$.

\noindent{\rm (a)} For $q\sim 0$,
$$
\mathcal{P}(q;\gamma) = - (c(\gamma))^\alpha + z_0q(1+o(1))\quad \text{where }z_0:= \frac{(m-\alpha)\left(h'(0) - \gamma c(\gamma)\right)}{(m-\alpha+1)(p-1)c(\gamma)}.
$$

\noindent{\rm (b)} For $q\sim \frac{m}{m-\alpha}$,
$$
\begin{array}{ll}
\displaystyle\mathcal{P}(q;\gamma) =   -C_{p,\gamma}^\alpha\Big( \frac{m}{m-\alpha}  -q\Big)   (1+o(1)),\ & p\ge2,\\[10pt]
\displaystyle\mathcal{P}(q;\gamma) =   -\left(\frac{|h'(1)|}{m\gamma}\Big( \frac{m}{m-\alpha}  -q\Big)\right)^{\alpha}   (1+o(1)),\ & p\in(1,2),\ \gamma>0.
\end{array}
$$
\end{Lemma}
Notice that if $p\in(1,2)$, then  $\mathcal{P}$ does not approach 0 linearly as $q\nearrow\frac{m}{m-\alpha}$. As a consequence, our next computations do not work in that range, and we have to assume $p\geq 2$.
\begin{Lemma}\label{lemma:c_differentiable}
Assume~\eqref{eq:slow.regime} and $p\geq 2$. The function $\gamma \to c(\gamma)$ belongs to $C^2([0,\gamma_*))$,  $\partial_\gamma\mathcal{P}\left(q;\gamma\right)$ exists and $\gamma \to \partial_\gamma\left(\mathcal{P}\left(\cdot;\gamma\right)\right)$ is continuous from $[0,\gamma_*)$ to $C([0, m/(m-\alpha)])$. Finally, $c'(\gamma)\in(-m,0)$.
\end{Lemma}
\begin{proof}
We show the computations only for $p>2$. The proof for the case $p=2$ can be adapted from~\cite[Lemma 2.5]{Du-Quiros-Zhou-Preprint} to deal with more general nonlinearities than the one considered there.

We define the incremental quotients
$$
\hat{\mathcal{Q}}_k(q) := \frac{|\mathcal{P}(q; \gamma+k)|^{p-1}-|\mathcal{P}(q; \gamma)|^{p-1}}{k},\quad \hat{c}_k:=\frac{c(\gamma+k) - c(\gamma)}{k},
$$
and we see that
\begin{equation*}
\begin{aligned}
&\hat{\mathcal{Q}}'_k(q) - a_k(q)\hat{\mathcal{Q}}_k(q)= \frac{\hat{c_k}}{(m-\alpha)q},\quad\text{where}\\
\displaystyle a(q):=&- \frac{1}{(m-\alpha)q} +  \frac{(\gamma + k)\left(|\mathcal{P}(q; \gamma+k)|^{p-2}-|\mathcal{P}(q; \gamma)|^{p-2}\right) +k |\mathcal{P}(q;\gamma)|^{p-2}}{|\mathcal{P}(q; \gamma+k)|^{p-1}-|\mathcal{P}(q; \gamma)|^{p-1}}\\[8pt]
 &+ \frac{\big(|\mathcal{P}(q; \gamma+k)|-|\mathcal{P}(q; \gamma)|\big)h(s(q))}{|\mathcal{P}(q; \gamma+k)|\,|\mathcal{P}(q; \gamma)|\,\left(|\mathcal{P}(q; \gamma+k)|^{p-1}-|\mathcal{P}(q; \gamma)|^{p-1}\right)s(q)},\\
s(q):=& \left(\frac{m-\alpha}{m}q\right)^{\frac{1}{m-\alpha}}.
\end{aligned}
\end{equation*}
By the Mean Value Theorem,
$$
\begin{aligned}
 a_k(q)=& -\frac{1}{(m-\alpha)q} +  \frac{(\gamma + k)(p-2)\theta_1^{p-3} +k |\mathcal{P}(q;\gamma)|^{p-2}}{(p-1)\theta_2^{p-2}}\\
&+ \frac{h(s(q))}{|\mathcal{P}(q; \gamma+k)|\,|\mathcal{P}(q; \gamma)|(p-1)\theta_3^{p-2}s(q)}
\end{aligned}
$$
for some $\theta_1, \theta_2, \theta_3\in\left( |\mathcal{P}(q; \gamma)|, |\mathcal{P}(q; \gamma+k)|  \right)$. Therefore,  if we let $k \to 0$,
$$
a_k(q)\to a(q):= -\frac{1}{(m-\alpha)q} +  \frac{ (p-2)\gamma}{(p-1)|\mathcal{P}(q; \gamma)|}  + \frac{h(s(q))}{(p-1) |\mathcal{P}(q; \gamma)|^p  s(q)}.
$$
Now we define $A_k(q)=\int_q^1 a_k(r)\,\textrm{d}r$, obtaining that
$$
\left( \textrm{e}^{A_k(q)} \hat{\mathcal{Q}}_k(q) \right)' =  \textrm{e}^{A_k(q)} \frac{\hat{c_k}}{(m-\alpha)q},
$$
and after a bit of work, similar to the one performed in~\cite[Lemma 2.5]{Du-Quiros-Zhou-Preprint}, we get
$$
\begin{aligned}
&
c'(\gamma) =-\displaystyle\frac{ (m-\alpha)\int_0^{\frac{m}{m-\alpha}}\Psi(q;\gamma)\,\textrm{d}q}{\int_0^{\frac{m}{m-\alpha}}q^{-1}\Psi(q;\gamma)\, \rm{d}q},\quad \text{with}
\\
&\Psi(q;\gamma):= q^{\frac{1}{m-\alpha}} \exp\left( \displaystyle\int_q^1 \Big(\frac{ (p-2)\gamma}{(p-1)|\mathcal{P}(r; \gamma)|}  + \frac{h(s(r))}{(p-1) |\mathcal{P}(r;\gamma)|^p s(r)}\Big) \, \rm{d}r \right).
\end{aligned}
$$
From this expression for $c'(\gamma)$ we immediately get $c'(\gamma)\in(-m,0)$. On the other hand, Lemma~\ref{lemma:limits_derivative} implies that $|\mathcal{P}|$ is a continuous function of~$\gamma$ in the $C^1([0, m/(m-\alpha)])$ norm, and therefore so is $|\mathcal{P}|^{p}$. Hence, $c'(\gamma)$ is also a continuous function of $\gamma$.
Going on we obtain, as in~\cite{Du-Quiros-Zhou-Preprint}, that
\begin{equation}\label{eq:formula_derivative_pressure}
\begin{aligned}
\partial_\gamma\left(|\mathcal{P}(q;\gamma)|^{p-1}\right) &=\, \frac{1}{\Psi(q;\gamma)}\int_0^q \Psi(r;\gamma) \frac{c'(\gamma)}{(m-\alpha)r} \, \rm{d}r \\[8pt]
  &=\, -  \frac{1}{\Psi(q;\gamma)}\int_q^{m/(m-\alpha)}\Psi(r;\gamma) \frac{c'(\gamma)}{(m-\alpha)r} \, \rm{d}r,
\end{aligned}
\end{equation}
from which  it is not hard to get
$$
\displaystyle\partial_\gamma\left(|\mathcal{P}(0;\gamma)|^{p-1}\right):=\lim\limits_{q\to 0^+} \displaystyle\partial_\gamma\left(|\mathcal{P}(q;\gamma)|^{p-1}\right) = c'(\gamma)
$$
using the behaviour of $\Psi$ when $q$ is near 0. Since we are in the case $p>2$, when $q\sim m/(m-\alpha)$  we have
$$
h(s(q))=\frac{h'(1)}{m}\left( q- \frac{m}{m-\alpha} \right)(1+o(1))\text{ and }|\mathcal{P}(q;\gamma)|=\frac{|h'(1)|}{c(\gamma)}\left( \frac{m}{m-\alpha}  -q\right)  (1+o(1)).
$$
Hence,
$$
\begin{aligned}
\Psi(q; \gamma)=&\,k_\sigma  \left(\frac{m}{m-\alpha} - q\right)^{\frac{(\alpha-1)\gamma c(\gamma)}{|h'(1)|}}\exp\left( -\sigma\left( \frac{m}{m-\alpha} - q\right)^{2-p} \right)(1+o(1)),\quad\text{with}\\
\sigma:=&\, \frac{(c(\gamma))^p}{(p-1)(p-2)m |h'(1)|^{p-1}},\\
k_\sigma:=&\,  \left(\frac{m}{m-\alpha}\right)^{\frac{1}{m-\alpha}} \left(\frac{m}{m-\alpha} - 1\right)^{\frac{(1-\alpha)\gamma c(\gamma)}{|h'(1)|}} \exp\left( \sigma \left(\frac{m}{m-\alpha} - 1 \right)^{2-p}\right).
\end{aligned}
$$
Introducing this in the second equality of~\eqref{eq:formula_derivative_pressure}, we can study the behaviour of $\partial_\gamma\left(|\mathcal{P}(q;\gamma)|^{p-1}\right)$ near $m/(m-\alpha)$. Since the integral is convergent, when $q\to m/(m-\alpha)$, the limit has the indeterminate form 0/0. Applying l'H\^opital's rule we  obtain that
$$
 \displaystyle\partial_\gamma\left(|\mathcal{P}(q;\gamma)|^{p-1}\right) \to 0\quad\text{when }q\to \frac{m}{m-\alpha}.
$$
Therefore $\gamma \to \partial_\gamma\left(\left|\mathcal{P}\left(\cdot;\gamma\right)\right|^{p-1}\right)$ is continuous from $\overline{\mathbb{R_+}}$ to $C([0, m/(m-\alpha)])$. With the help of the formulas above it is easy  to prove,  as in~\cite{Du-Quiros-Zhou-Preprint}, that $c''(\gamma)$ exists and is continuous as a function of $\gamma$. Using that
$$
\displaystyle\partial_\gamma\left(|\mathcal{P}(q;\gamma)|^{p-1}\right) = -(p-1)|\mathcal{P}(q;\gamma)|^{p-2}\partial_\gamma\left(\mathcal{P}(q;\gamma)\right),
$$
together with Lemma~\ref{lemma:limits_derivative}, it is also easy to check that $\gamma \to \partial_\gamma\left(\mathcal{P}\left(\cdot;\gamma\right)\right)$ is continuous from $\overline{\mathbb{R_+}}$ to $C([0, m/(m-\alpha)])$.
\end{proof}

\noindent\emph{Remark. } Using the estimate for $c'(\gamma)$, we get that $c(\gamma)\ge c(0)-m\gamma$. Then, through a continuity argument, it is possible to show that $c(\gamma)$ is well defined at least for $\gamma\in [0, c(0)/m)$.

\medskip

The last lemma of this part is similar to the one in~\cite[Lemma 2.7]{Du-Quiros-Zhou-Preprint}.
\begin{Lemma}
Assume~\eqref{eq:slow.regime} and $p\geq 2$.  For any $L>0$ there exists a constant $C=C_L>0$ such that
$$
|\partial_\gamma V(x; \gamma)| \leq C\quad\text{for }x\leq 0,\ \gamma\in [0,L].
$$
\end{Lemma}


\section{Spreading}
\label{sect-propagation.vs.vanish} \setcounter{equation}{0}

The first aim of this section is to give conditions guaranteing that there are solutions that spread. The first one is a condition on the reaction nonlinearity, showing that if it is strong enough at the level zero then all nontrivial (nonnegative) solutions spread. This was called in the semilinear case the \textit{hair-trigger effect}~\cite[Chapter 3]{Aronson-Weinberger-1978}. The second one is a sufficient condition on the initial datum, valid for all reaction nonlinearities satisfying~\eqref{eq:reaction}. The second aim is to show that the aymptotic spreading speed, when spreading happens, is given by the velocity $c^*$ of the unique finite wavefront. These results extend to the slow diffusion regime the results developed for the semilinear case by Aronson and Weinberger in their famous work~\cite{Aronson-Weinberger-1978}.

\subsection{The hair-trigger effect}

We will now prove Theorem~\ref{thm:spreading.vs.vanishing}, which says that if the reaction at the level zero is strong enough, then all nontrivial nonnegative solutions spread. The proof is based on the fact that all nontrivial nonnegative solutions to
\begin{equation}
\label{eq:Fujita}
u_t=\Delta_p u^m+ku^{q_F},
\end{equation}
where $k>0$ and $q_F:=m(p-1)+p/N$ is the so-called Fujita exponent,
blow up in finite time; see~\cite{Galaktionov-Levine-1998}.
Since we will work in parallel with~\cite[Section 3]{Aronson-Weinberger-1978} and~\cite[Section 4]{Garriz-2020}, we will be rather sketchy, and give details of the proofs only  when we consider it necessary.

The first step is to show that $\|u(\cdot,t)\|_{L^\infty(\mathbb{R}^N)}$ approaches 1 along some sequence of times.
\begin{Lemma}\label{lema11}
	Under the hypotheses of   Theorem~\ref{thm:spreading.vs.vanishing},
$
\displaystyle\limsup\limits_{t \rightarrow \infty} \|u(\cdot,t)\|_{L^\infty(\mathbb{R}^N)} =1$.
\end{Lemma}
The proof of this statement, which is analogous to the one of~\cite[Lemma 3.2]{Aronson-Weinberger-1978}, uses a comparison argument with a solution of problem~\eqref{eq:Fujita}.  Note that this is not enough to obtain convergence to~1 in compact sets: our solution could behave like a \lq\lq spike'', or approach 1 in a set that travels to infinity.

The following lemma is based in the work of Kanel'~\cite{Kanel-1964}. We already presented a sketch of the proof for the case $p=2$ in~\cite[Lemma 4.2]{Garriz-2020}, commenting the principal differences with the one in~\cite[Lemma 3.3]{Aronson-Weinberger-1978} for the case $m=1$, $p=2$. Exploring the case $p\neq2$ a simpler proof has been found.

\begin{Lemma}\label{lema13}
	Let $\hat h$ be a reaction nonlinerity satisfying~\eqref{eq:reaction} with $a=0$ and such that
	$$
	\hat h(\hat u)=k \hat u^{q_F}\quad \text{for }\hat u\in [0,b]
	$$
	for some $b\in(0,1)$  and $k>0$. 		
Let $\hat u$ denote the solution of the problem
	$$
	\hat u_t = \Delta_p \hat u^m + \hat h(\hat u),\qquad \hat u(x,0) =\begin{cases} \delta (1-|x|^\sigma)^\beta\quad&\text{if } 0 \le |x| \leq 1,\\
	0\quad&\text{if }|x|>1,
	\end{cases}
$$
for some 	
$\delta\in(0,1)$, $\sigma>\frac{p}{p-1}$ and $\beta>0$. If $\sigma$ and $\beta$ are big enough, then
$\lim\limits_{t \rightarrow \infty} \hat u(0,t)  =1$.
\end{Lemma}

\begin{proof}
Since the initial datum is radially symmetric and nonincreasing in $|x|$, it is easy to check through well-known comparison arguments that the same is true for $\hat u(\cdot,t)$ for  all $t>0$. Hence,
\begin{equation}
\label{eq:first.step}
0 \leq \hat u(x,t) \leq \hat u(0,t) \leq 1.
\end{equation}

The next step is to show that $\hat u(0,t)$ is nondecreasing, so that its limit exists.  To this aim we define 
$\underline u(x):=\hat u(x,0)$. Then
\[
\underline u_t-\Delta_p\underline u^m-\hat h(\underline u)=-\Delta_p\underline u^m-\hat h(\underline u)=Z,
\]
with
\begin{gather*}
Z(x):=\begin{cases}-\big(\delta(1-|x|^\sigma)^\beta \big)^{m(p-1)}\left\lbrace C_1(1-|x|^\sigma)^{\beta p /N} + C_2 \frac{|x|^{p(\sigma-1)}}{(1-|x|^\sigma)^p} - C_3 \frac{|x|^{p(\sigma-1)-\sigma}}{(1-|x|^\sigma)^{p-1}}\right\rbrace&\text{if } |x|\le 1,\\
0&\text{if } |x|>1,
\end{cases}
\\
C_1=k\delta^{p/N},\ C_2=\sigma(\beta m -1)(p-1)(\sigma\beta m)^{p-1}\text{ and } C_3=\big((p-1)(\sigma-1)+N-1 \big)(\sigma\beta m)^{p-1}.
\end{gather*}
If $\beta> p/(m(p-1))$ and $\sigma> p/(p-1)$, then $Z$ is continuous. Taking $\beta$ big enough (and satisfying the previous condition), we can make $C_2\ge C_3$, and hence $Z(x)\leq 0$ for all $x$. Moreover, for such $\beta$, $\underline u$ is smooth.
Thus $\underline u$ is a lower solution to the initial value problem satisfied by $\hat u$. The comparison principle then implies $\underline u(x)\leq \hat u(x,t)$ for all $t\geq 0$.
Take an arbitrary constant $s>0$.  Clearly $(x,t)\to \hat u(x, s+t)$ is the unique solution of the initial value problem
\[
u_t=\Delta_p u^m+h(u),\;\; u(x,0)=\hat u(x,s).
\]
Since $\hat u(x, s)\geq \underline u(x)=\hat u(x,0)$, by  the comparison principle again we deduce that $\hat u(x, s+t)\geq \hat u(x, t)$ for all $t\geq 0$ and $x\in \mathbb R^N$; that is, $\hat u(x, t)$ is nondecreasing in $t$. Thus
$
\eta^* := \lim_{t \rightarrow \infty} \hat u(0,t)
$
exists. 
According to \eqref{eq:first.step},
$0 \leq \hat u(x,t) \leq \hat u(0,t) \le \eta^*$ for all $x\in\mathbb{R}^N$, $t>0$,
which is a contradiction with Lemma~\ref{lema11} if $\eta^* <1$.
\end{proof}
The proof of Theorem~\ref{thm:spreading.vs.vanishing} now follows easily,  reasoning as in~\cite[Theorem 3.1]{Aronson-Weinberger-1978}.

\subsection{Conditions on the initial data}\label{Over the Fujita exponent}

We are now ready to prove that, given any reaction nonlinearity satisfying~\eqref{eq:reaction}, solutions to~\eqref{eq:main} whose initial datum is big enough spread.

Let us fix $\rho>0$. This will be the radius of a ball where a certain radial function is flat. We will take the convection parameter $\gamma = (N-1)/\rho$. With $\rho$ big enough,  $\gamma\in[0,\gamma_*)$ and so $c(\gamma)>0$ is well defined. We select now $c \in (0,c(\gamma))$ for the speed of propagation. Given $\eta\in(\beta_c,1)$,  let $U^{c,\eta}$ be the function provided by Lemma~\ref{lem:profile.subsolution}.  We define
$$
\mathfrak{u}_0(x;\rho,c,\eta)=\mathfrak{U}(|x|;\rho,c,\eta), \qquad\text{where }\mathfrak{U}(r;\rho,c,\eta):= \begin{cases}
             \eta,\quad &r\leq \rho,
             \\
             U^{c,\eta}(r - \rho),\quad &\rho < r\leq \rho + b,
             \\
             0,\quad &r\geq \rho + b.
             \end{cases}
$$
\begin{Lemma}\label{lema_propagation}
Let $\mathfrak{u}$ be the solution of \eqref{eq:main} with initial datum $\mathfrak{u}(x,0)=\mathfrak{u}_0(x;\rho,c,\eta)$. Then
$\lim\limits_{t \rightarrow \infty} \mathfrak{u}(x,t) = 1$
uniformly in compact subsets of~$\mathbb{R}^N$, and $\mathfrak{u}(x,t) \geq \eta$ for $|x| \leq \rho + c t$, $t \geq 0$.
\end{Lemma}

\begin{proof}
We choose an arbitrary $c_1\in (0,c)$ and define
$W(x,t) = \mathfrak{U}(|x| - c_1t;\rho,c,\eta)$. Denoting  $\xi=|x|-\rho-c_1t$, a direct computation shows that
$$
\begin{aligned}
\mathcal{L} W(x,t)  =&\, -c_1(U^{c,\eta})'(\xi) - \big(|((U^{c,\eta})^m)'|^{p-2}((U^{c,\eta})^m)'\big)' (\xi)\\
&+ \frac{N-1}{|x|}|((U^{c,\eta})^m)'|^{p-1}(\xi) - h(U^{c,\eta})(\xi) \\
\le&\, (c_1-c)|(U^{c,\eta})'|(\xi)\le 0\qquad\text{if }\rho + c_1t < |x|< \rho + b + c_1t,
 \end{aligned}
$$
where we have used that $|x|>\rho$ and the equation,~\eqref{eq:profile.subsolution}, satisfied by $U^{c,\eta}$.
On the other hand, $\mathcal{L}W(x,t)=-h(\eta)<0$ if $|x|< \rho + c_1t$, since $\eta\in(\beta_c,1)$, and, trivially, $\mathcal{L}W(x,t)=0$ if $|x|> \rho + b + c_1t$. Summarizing, $\mathcal{L}W\le0$ if $|x|\neq \rho + c_1t,\rho + b + c_1t$.

We finally observe that the flux $|\nabla W^m|^{p-2}\nabla W^m$ is continuous across the hypersurface $|x|= \rho + c_1t$, and that, though it is not continuous across  the hypersurface $|x|= \rho + b + c_1t$, it has a negative jump when we cross it from the inside to the outside. All together is enough to prove that~$W$ is a weak subsolution for our problem. Since, by comparison, $W(x,t) \leq \mathfrak{u}(x,t)$ for $x\in \mathbb{R}^N$, $t\geq 0$ the second assertion of the lemma follows.

On the other hand,  since $\mathfrak{u}(x,h)\geq W(x,h) \geq W(x,0)=\mathfrak{u}(x,0)$ for any $h>0$ we have, comparing $\mathfrak{u}(x,t+h)$ and $\mathfrak{u}(x,t)$, that $\mathfrak{u}$ is a nondecreasing  function of time that is bounded above by the value 1. Thus, $\mathfrak{u}(x,t)$ converges uniformly on compact sets to a certain $\tau(x)\in[\eta,1]$. At this point it is easy to prove, as in~\cite[Lemma 5.1]{Aronson-Weinberger-1978}, that necessarily $\tau(x) = 1$.
\end{proof}

We can now assert the main result of this subsection, from which Theorem~\ref{thm:spreading} follows.

\begin{Theorem}\label{tma_over_fujita}
Let $u$ be a solution of~\eqref{eq:main} such that
$u(x,0) \geq \mathfrak{u}_0(x-x_0;\rho,c,\eta)$ for some $x_0 \in \mathbb{R}^N$, and admissible $\rho,c,\eta >0$.
Then, for any $y_0 \in \mathbb{R}^N$,
$$
\begin{array}{ll}
\lim\limits_{t \rightarrow \infty} \min\limits_{|y - y_0| \leq ct} u(y,t) = 1 \qquad &\text{if }c \in [0,c^*),\\
\lim\limits_{t \rightarrow \infty} u(y,t) = 0\quad\text{for }|y - y_0| \geq ct\quad&\text{if }c>c^*.
\end{array}
$$
\end{Theorem}

\begin{proof}

The first assertion comes from Lemma~\ref{lema_propagation}, arguing as in the proof of~\cite[Lemma 5.3]{Aronson-Weinberger-1978}.

As for the second assertion, the first step is to prove, by comparison with a flat (constant in space) solution, that $\|u(\cdot,t)\|_{L^\infty(\mathbb{R}^N)}<1+\delta$ for any $\delta>0$ if $t$ is large enough. The result then follows by comparison with one-dimensional  supersolutions of the form $w(x,t;\nu)=f(t)U_{c^*}(x\cdot\nu-g(t))$, where $\nu$ is an arbitrary unitary vector; see Subsection~\ref{subsect-one.dimensional}. These supersolutions travel  asymptotically with speed $c^*$ and have a finite front, hence the result.
\end{proof}
\section{Convergence}
\label{sect-convergence.compact.support} \setcounter{equation}{0}

The results in the previous section show that the asymptotic spreading speed is $c^*$. However, they only give a very rough  estimate  on the asymptotic location of the level sets, and none on the asymptotic form of the solution close to the free boundary. This section is devoted to these issues. On the rest of the article we will only consider the case $p\geq 2$,  together with the usual  slow-regime condition~\eqref{eq:slow.regime}.

\subsection{Logarithmic shift}
We start this subsection showing that solutions approach 1 exponentially in certain expanding sets.

\begin{Lemma}\label{lemma:exponential-convergence-to-1}
Let $u$ be a spreading solution of equation~\eqref{eq:main} with a bounded and compactly supported initial datum. There exist $\hat c\in(0, c^*)$, $k>0$ and $M,T_*>0$
such that
\begin{eqnarray}
\label{u<} u(x,t)\leq 1+M {\rm e}^{-k t},&\qquad &x\in\mathbb{R}^N,\ t\geq 0,\\
\label{u>} u(x,t)\ge 1-M {\rm e}^{-k t},&\qquad &|x|\le \hat ct,\ t\ge T_*.
\end{eqnarray}
\end{Lemma}
\begin{proof}
 Assertion \eqref{u<} can be easily proved by comparison with a flat  ODE solution thanks to the fact that $h'(1)<0$. Assertion \eqref{u>} is a bit more complicated.

When $p=2$ (which implies that $m>1$ since $m(p-1)>1)$) we can use an  argument similar to the one in~\cite[Lemma 3.2]{Du-Quiros-Zhou-Preprint}, recalling that $h(u)\sim |h'(1)|(1-u)$ when $u\sim 1$.  The argument involves certain linearization that does not work for $p\neq2$.

For $p>2$ we consider the radial comparison function
$$
v(x,t; T):=  \varphi(|x|)\left(1-  \varepsilon e^{-\lambda(t-T)}\right),\quad \varphi(|x|):= \begin{cases}
            1,\quad &0\leq |x|\leq \frac{\tilde{c}T}{2},
             \\
             1-\varepsilon \left(\frac{2}{\tilde{c}T}\right)^n\left(|x| - \frac{\tilde{c}T}{2}\right)^n,\quad & \frac{\tilde{c}T}{2}< |x| \leq \tilde{c}T,
             \end{cases}
$$
for some speed $\tilde{c}\in (0,c^*)$, and certain positive constants $T,\lambda,n,\varepsilon$ that will be chosen so that $v$ is a subsolution of the problem in the ball of radius $R_T:=\tilde{c}T$ for all $t> T$. We omit most of the dependencies of $v$ to simplify the  notation. Observe that $\varphi\in C^2(B_{R_T})$ if $n\ge 2$.

Notice that $1>v(x,t)\ge (1-\varepsilon)^2$ if $|x|\le R_T$, $t> T$. Therefore, since $h(v)\sim (1-v)|h'(1)|$ when $v\sim 1$, there exists $\varepsilon_0\in(0,1)$ such that for all $\varepsilon\in(0,\varepsilon_0)$ we have, denoting $r:=|x|$,
$$
h(v)\ge\frac{|h'(1)|}{2} (1-v)=\frac{|h'(1)|}{2}\left[\varepsilon \left(\frac{2}{R_T}\right)^n\left(r - \frac{R_T}{2}\right)_+^n + \varphi(r)\varepsilon {\rm e}^{-\lambda(t-T)}\right].
$$
Hence, since $v_t = \varphi(r)\varepsilon \lambda {\rm e}^{-\lambda(t-T)}$, it is trivial to have $\mathcal{L}v\leq 0$  if $r\le R_T/2$ and $t>T$, simply by choosing $\lambda\leq |h'(1)|/2$.

For $r\in(R_T/2,R_T)$ we have
$$
\begin{array}{l}
\frac{N-1}{r}\big(|(\varphi^m)'|^{p-2}(\varphi^m)'\big)(r) =-\frac{N-1}{r}\left[m(\varphi(r))^{m-1}n \varepsilon \left(\frac{2}{R_T} \right)^n \left(r-\frac{R_T}{2} \right)^{n-1}\right]^{p-1},\\[10pt]
\big(|(\varphi^m)'|^{p-2}(\varphi^m)'\big)'(r)=-(p-1)\left[m(\varphi(r))^{m-1}n \varepsilon \left(\frac{2}{R_T} \right)^n \left(r-\frac{R_T}{2} \right)^{n-1}\right]^{p-2}\times\\[10pt]
\times \left[m(m-1)(\varphi(r))^{m-2}n^2 \varepsilon^2 \left(\frac{2}{R_T} \right)^{2n} \left(r-\frac{R_T}{2} \right)^{2n-2} + m(\varphi(r))^{m-1}n(n-1) \varepsilon \left(\frac{2}{R_T} \right)^{n} \left(r-\frac{R_T}{2} \right)^{n-2} \right].
\end{array}
$$
which means that there are positive constants $C_1, C_2$, which may depend on $n$, but not on $R_T$, such that
$$
\begin{array}{l}
\frac{N-1}{r}\big(|(\varphi^m)'|^{p-2}(\varphi^m)'\big)(r)  \geq -C_1 R_T^{-(1+n(p-1))}\left(r-\frac{R_T}{2} \right)^{(n-1)(p-1)},\\[8pt]
\big(|(\varphi^m)'|^{p-2}(\varphi^m)'\big)'(r)\geq-C_2R_T^{-n(p-1)}\left(r-\frac{R_T}{2} \right)^{n-2+(n-1)(p-2)}.
\end{array}
$$
Therefore, if $n > \max\{ 2, p/(p-2) \}$  (it is at this point where we need $p>2$), the reaction term is able to take care of both the time derivative term and the terms coming from the doubly nonlinear operator, by taking  $T\ge T_1$ for some large $T_1$, so that $R_T$ is as big as needed. Thus, $\mathcal{L}v\leq 0$ also for $R_T/2\le |x|<R_T$, $t\ge T$, and $v$ is a subsolution, as desired.

We still have to check that $v$ is below $u$ at the parabolic boundary. This comes from Theorem~\ref{thm:spreading}. Indeed, given $\varepsilon\in (\varepsilon_0,1)$, there exists a time $T_2\ge T_1$ such that if $T\ge T_2$, then
$$
\begin{array}{ll}
u(x,T)\geq 1-\varepsilon\geq v(x, T; T) \quad& \text{for all }|x|\leq R_T,\\[8pt]
u(x, t)\geq 1-\varepsilon\geq v(x, t; T) \quad &\text{for all }|x|= R_T,\ t\geq T.
\end{array}
$$
Comparison then implies $u(x,t)\ge v(x,t;T)$ if $|x|\le R_T$, $t\ge T\ge T_2$. In particular,
$$
u(x,2T)\geq v(x,2T; T) = 1-\varepsilon {\rm e}^{-\frac{\lambda}{2}\cdot 2T} \quad\text{for all }|x|\leq 2T\left(\frac{\tilde{c}}{4}\right),\ T\geq T_2,
$$
and~\eqref{u>} follows, taking
$\hat c= \tilde{c}/4$, $k=\lambda/2, M=\varepsilon$ and $T_*=2T_2$.
\end{proof}

Next, a lemma that will allow us to establish comparison in the range $r\geq c^*t - M \log t$, for a certain $M>0$, in the future. The proof is very similar to the one of~\cite[Lemma 3.3]{Du-Quiros-Zhou-Preprint}, so we will  only explain the major differences.

\begin{Lemma}
\label{lem:first.log.correction}
There exist positive constants $T>0$ and $M>0$ such that
$$
c^* -M \log(t)  \leq \eta(t)\quad \text{for } t\geq T.
$$
\end{Lemma}
\begin{proof}

Following ideas from~\cite[Lemma 3.3]{Du-Quiros-Zhou-Preprint} we look for a subsolution of the form
$$
w(x,t)=f(t)U_{c^*}(|x|-g(t) + M\log t)
$$
for a certain positive $M$, where $f$ and $g$ are chosen as in Subsection~\ref{subsect-one.dimensional}.
We have to check that $\mathcal{L}w\leq 0$ in the range of comparison $r>\hat{c}t$, where this $\hat{c}$ is the one from the previous lemma. In order to see this we compute
$$
\begin{aligned}
\mathcal{L}w &=  f'U_{c^*} - fg'U'_{c^*}  +  fU'\frac{M}{t} + f^{m(p-1)}(|(U_{c^*}^m)'|^{p-1})' + f^{m(p-1)}\frac{N-1}{r}|(U_{c^*}^m)'|^{p-1}  - h(fU_{c^*})\\ \nonumber
& \leq\underbrace{ \varphi U_{c^*} + k \varphi U'_{c^*} + f^{m(p-1)}h(U_{c^*}) - h(fU_{c^*})}_{\mathcal{A}} + \underbrace{fU'_{c^*}\frac{M}{t} + f^{m(p-1)}\frac{N-1}{r}|(U_{c^*}^m)'|^{p-1} }_{\mathcal{B}}.
\end{aligned}
$$
Now the term $\mathcal{A}$ can be handled by choosing the correct function $\varphi$ and the correct value of $k$ as in~\cite{Garriz-2020} to prove that $\mathcal{A}\leq 0$. On the other hand,
$$
\mathcal{B}\leq \frac{f|U'_{c^*}|}{t}\left(  \frac{(N-1)f^{m(p-1)-1} (mU^{m-1}_{c^*})^{p-1}|U'_{c^*}|^{p-2} }{\hat{c}} - M \right),
$$
and since $U^{(m-1)(p-1)}_{c^*}|U'_{c^*}|^{p-2} \to 0$ as $U_{c^*}\to 1$ and as $U_{c^*}\to 0$ (we recall~\eqref{profile:property}) it is clear that we can find a big enough $M$ such that $\mathcal{B}\leq 0$. The rest of the result follows as in~\cite[Lemma 3.3]{Du-Quiros-Zhou-Preprint}.
\end{proof}

\noindent\emph{Remark.  } There is a minor mistake in the computation of $\mathcal{L}w $ in the proof of~\cite[Lemma 3.3]{Du-Quiros-Zhou-Preprint}. The term \lq\lq $-Mg(t)\textrm{e}^{-\delta t}\Phi_{ c_*}$'', which is exploited there to show that $\mathcal{L}w \leq 0$, does not exist.  Nonetheless, the result is true, as we have proved above.

\medskip

\begin{Theorem}
There exist positive constants $T$ and $C$ such that
$$
c^* - (N-1)c_{\sharp} \log(t) - C \leq \eta(t) \leq c^* - (N-1)c_{\sharp} \log(t) + C\quad \text{for }t\geq T.
$$
\end{Theorem}

\begin{proof}\label{Sketch of the proof}
Let us focus on the left-hand side of the inequality, since the right-hand one comes similarly. As in~\cite[Lemma 3.4]{Du-Quiros-Zhou-Preprint} we define a subsolution by
$$
w(r,t):= f(t)U\left( |x|-g(t)  + C  ; \gamma\left(c^*- \frac{(N-1)c_\sharp}{t}\right) \right),
$$
where $\gamma(c)$ is the inverse of the function $c(\gamma)$ (which is a well defined $C^2$ function for $\gamma\in[0,\gamma_*)$, thanks to the results of Subsection~\ref{subsect-dependence.convection}), and $t_0$ and $C$ are suitably chosen constants; see~\cite{Du-Quiros-Zhou-Preprint} for the details.  Here we take
$$
f(t):= 1-\frac{\log (t-t_0)}{(t-t_0)^2}
$$
and $g(t)$ as the solution of the ODE
$$
g'(t):=f^{m(p-1)-1}c^*  + (1-f^{m(p-1)-1})\frac{(N-1)c_\sharp}{t} - \frac{\lambda\log(t-t_0)}{(t-t_0)^2} + \frac {\lambda}{(t-t_0)^2}.
$$

In order to apply comparison,
we want to check if $\mathcal{L}w \leq 0$ in the region $r\geq c^*t - M\log t$, where $M$ is the constant from the previous lemma. This is the same as checking
$$
\begin{aligned}
\mathcal{L}w =&\, U_\gamma   f\gamma^\prime\frac{(N-1)c_\sharp}{t^2} + U' f \left[ f^{m(p-1)-1} \left(c^* - \frac{(N-1)c_\sharp}{t}\right) - g' + \frac{(N-1)c_\sharp}{t}\right]\\
& + |(U^m)'|^{p-1}f^{m(p-1)} \left[\frac{N-1}{r} - \gamma\right]+ U\left[f' + \frac{f^{m(p-1)}h(U) - h(fU)}{U}\right]\\
 =&\,U'f \left( \frac{\lambda\log(t-t_0)}{(t-t_0)^2} - \frac {\lambda}{(t-t_0)^2}\right) + |(U^m)'|^{p-1} f^{m(p-1)}\left( \frac{(N-1)}{r} - \gamma \right)\\
& +  U_\gamma f\gamma^\prime\frac{(N-1)c_\sharp}{t^2} + Uf' +f^{m(p-1)}h(U) - h(fU) \leq 0,
\end{aligned}
$$
where $U_\gamma$ is the derivative of $U$ with respect to the second parameter and the dependence on $t$ of $f$ and $g$ has been omitted. Since multiplying the inequality by a positive quantity does not affect its sign, we multiply the previous equation by $(m-\alpha)U^{m-\alpha-1}$ to obtain
$$
\begin{array}{l}
(m-\alpha)U^{m-\alpha-1}\mathcal{L}w = (U^{m-\alpha})_\gamma\underbrace{ \left[f\gamma^\prime\frac{(N-1)c_\sharp}{t^2}\right]}_{\mathcal{A}}\\
 \qquad+ (m-\alpha)U^{m-\alpha+1}\underbrace{\left[f' +\frac{(f^{m(p-1)}-1)h(U)}{U} + \frac{h(U)-h(fU)}{U}\right]}_{\mathcal{B}}\\
 \qquad+  (U^{m-\alpha})'\underbrace{\left[      \frac{\lambda\log(t-t_0)}{(t-t_0)^2} - \frac {\lambda}{(t-t_0)^2}       - (mU^{m-1})^{p-1}|U'|^{p-2} f^{m(p-1)}\left( \frac{(N-1)}{r} - \gamma \right)\right]}_{\mathcal{C}}.
\end{array}
$$
Dealing as in~\cite[Lemma 3.4]{Du-Quiros-Zhou-Preprint} and using the results in Subsection~\ref{subsect-dependence.convection}, we see that the term $\mathcal{A}$ is $O(t^{-2})$, so let us focus on the other two.

Making use of the Mean Value Theorem, we see that the term $\mathcal{B}$ satisfies
\begin{align}\notag
\mathcal{B} &= \frac{2\log(t-t_0)}{(t-t_0)^3} - \frac{1}{(t-t_0)^3} -m(p-1)\frac{\log (t-t_0)}{(t-t_0)^2}\theta+ O(1/t^2)  +\frac{\log (t-t_0)}{(t-t_0)^2}h'(\nu)\\ \nonumber
&= \left( h'(\nu)-m(p-1)\theta\right)\frac{\log (t-t_0)}{(t-t_0)^2}+ O(1/t^2)
\end{align}
for  some  $\nu\in (0,1)$ and $\theta\in\mathbb{R}$. On the other hand, since $U^{(m-1)(p-1)}_{c^*}|U'_{c^*}|^{p-2} \to 0$ as $U_{c^*}\to 1$ and as $U_{c^*}\to 0$, there must exist a positive constant $C_1$ such that
$$
\mathcal{C}\geq \frac{\lambda\log(t-t_0)}{(t-t_0)^2} - \frac {\lambda}{(t-t_0)^2}   - C_1\left( \frac{(N-1)}{r} - \gamma \right).
$$
Therefore, since $\frac{(N-1)}{r} - \gamma \leq \frac{(N-1)M\log t}{(c^* t)^2} +O(1/t^2)$, see~\cite{Du-Quiros-Zhou-Preprint},  choosing $\lambda$ big enough, there is a constant $\mu>0$ such that
$$
\mathcal{C}\geq\frac{\mu\log(t-t_0)}{(t-t_0)^2} + O(1/t^2).
$$
Note that we can make $\mu$ as big as we want just by taking $\lambda$ bigger.

In the end, we arrive  at
$$
\begin{array}{lcc}
(m-\alpha)U^{m-\alpha-1}\mathcal{L}w \leq (U^{m-\alpha})_\gamma O(1/t^2)\\[10pt] \nonumber
 \displaystyle\qquad+\frac{\log(t-t_0)}{(t-t_0)^2}\left(  \left( h'(\nu)-m(p-1)\theta\right) (m-\alpha)U^{m-\alpha+1}  - \mu|(U^{m-\alpha})'| \right),
\end{array}
$$
and therefore it is enough to check that
\begin{equation}\label{eq:ineq_shift}
 \left( h'(\nu)-m(p-1)\theta\right)(m-\alpha)U^{m-\alpha+1}  - \mu|(U^{m-\alpha})'|\leq 0,
\end{equation}
since by the results in Subsection~\ref{subsect-dependence.convection} we have that $(U^{m-\alpha})_\gamma\sim O(1)$. Since $h'(1)<0$, there is a value $\delta>0$ such that $h'(U)<0$ for all $U\in (1-\delta, 1]$ and $\theta$  close enough to 0.
If $U\in (1-\delta/2, 1]$, we take $T>0$ big enough such that for $t>T$ we have that $fU\in(1-\delta, 1]$ and therefore $h'(\nu)<0$ and $\theta\sim 0$, so that~\eqref{eq:ineq_shift} is satisfied. On the other hand, if $U\in [0, 1-\delta/2]$, since $|(U^{m-\alpha})'| \to (c^*)^\alpha$ as $U\to 0$, it is easy to see that we can choose a big enough $\mu$ (by making $\lambda$ big) such that again we have~\eqref{eq:ineq_shift}, checking therefore that $\mathcal{L}w\leq 0$.
\end{proof}

\subsection{Convergence}

We still need one more tool to prove~Theorem~\ref{thm:convergence.compact.support}, a characterization of certain solutions for the problem in moving coordinates defined for \emph{all} times, and not only from a certain time on. Such solutions are known as \emph{eternal solutions}.
\begin{Theorem}\label{tma:eternal}
Let $\mathcal{U}=\mathcal U(\xi, t)$ be a nonincreasing (in the variable $\xi$) weak solution to
\begin{equation}
\label{eq:eternal}
\mathcal{U}_t=\Delta_p \mathcal{U}^m + c^*\mathcal{U}^\prime+h(\mathcal{U}),\quad (\xi,t)\in\mathbb{R}^2
\end{equation}
such that
$$
U_{c^*}(\xi+C)\leq \mathcal{U}(\xi,t)\leq U_{c^*}(\xi-C)
$$
for some $C >0$. Then there exists a constant $\xi^*\in[-C,C]$ such that $\mathcal{U}(\xi,t)=U_{c^*}(\xi - \xi^*)$.
\end{Theorem}
We are saying that every eternal solution $\mathcal{U}$ of equation~\eqref{eq:eternal} that is nonincreasing and trapped between two profiles is, actually, a wavefront. Thus, if $u$ converges to a eternal solution that satisfies the hypotheses of Theorem~\ref{tma:eternal}, it actually converges to a stationary profile. The proof of this theorem is analogous to the ones for the case $p=2$ given in detail in~\cite{Du-Quiros-Zhou-Preprint,Garriz-2020}.

With all these tools at hand, the proof of Theorem~\ref{thm:convergence.compact.support} is very similar to the one for the case $p=2$ and logistic nonlinearity given in~\cite[Section 6]{Du-Quiros-Zhou-Preprint}.

\medskip

\noindent\emph{Sketch of the proof of Theorem~\ref{thm:convergence.compact.support}. } First we prove that $\limsup\limits_{t\to\infty} u(\eta(t) - r, t)>0$ for every $r>0$. This, together with the identification of eternal solutions and the results about the logarithmic shift and the bound for the flux of the previous sections allow us to prove, similarly to~\cite[Section 6.1]{Du-Quiros-Zhou-Preprint}, that there exists a sequence $\{t_n\}\to \infty$ and a constant $r_0\in \mathbb{R}$ such that, for any $a\in \mathbb{R}$,
$$
\begin{aligned}
&\lim\limits_{n\to \infty} u(r+c^*t_n - (N-1)c_\sharp \log t_n) = U_{c^*}(r-r_0)\quad\text{uniformly in }r>a,\\ &\lim\limits_{n\to\infty} (\eta(t_n) - c^*t_n + (N-1)c_\sharp \log t_n) = r_0.
\end{aligned}
$$
The last part of the proof consists in showing that this convergence is not restricted to a specific time sequence. This can be done following~\cite[Section 6.2]{Du-Quiros-Zhou-Preprint}.
\qed

\subsection{The one-dimensional case}
\label{subsect-one.dimension}

When $N=1$ we can give a precise description of the long-time behaviour for spreading solutions with bounded and compactly supported initial data without having to ask them to be symmetric.  Indeed, if we define the left and right free boundary as
$$
\zeta_+(t) := \inf \{r \in \mathbb{R} : u(x,t)=0\text{ for all }x \geq r\}, \qquad
\zeta_-(t) := \sup \{r \in \mathbb{R} : u(x,t)=0\text{ for all }x \leq r\},
$$
following the steps from~\cite[Section 5]{Garriz-2020} it is easy to obtain  the following result.
\begin{Theorem}
Let  $N=1$. Let $u$ be a spreading solution of problem~\eqref{eq:main} corresponding to a bounded and compactly supported initial datum $u_0$, and let $\zeta_\pm$ be its left and right free boundaries. There exist constants $\xi_\pm\in\mathbb{R}$ such that
$$
\begin{aligned}
\lim\limits_{t \rightarrow \infty} \sup\limits_{x \in \overline{\mathbb{R}_+}} |u(x+c_*t,t) - U_{c^*}(x-\xi_+)| = 0, \qquad\lim\limits_{t\to\infty}\zeta_+(t)-c^*t=\xi_+,\\
\lim\limits_{t \rightarrow \infty} \sup\limits_{x \in \overline{\mathbb{R}_-}} |u(x-c_*t,t) - U_{c^*}(\xi_--x)| = 0, \qquad\lim\limits_{t\to\infty}\zeta_-(t)+c^*t=\xi_-.
\end{aligned}
$$
\end{Theorem}

In this one-dimensional setting we are also able to deal with certain nonnegative initial data  whose supporting sets  are unbounded in one direction of the real line.

Let $\delta>0$ be  such that $h^\prime(u)<0$ for $u\in (1-\delta,1+\delta)$; see~\eqref{eq:reaction_deriv_in_1}.
We define the class of functions
 \begin{equation*}\label{def:class_A}
 \mathcal{A}:=\{ v \geq 0 : \|v\|_\infty<\infty,  v(x) \equiv 0\text{ for all }x \geq \bar x \text{ for some } \bar x \in \mathbb{R}, \ \liminf_{x \rightarrow -\infty} v(x) \in (1-\delta,1+\delta)\}.
 \end{equation*}
 Solutions with initial data in this class have a \emph{right free boundary},
\begin{equation*}\label{eq:zeta}
\zeta(t) := \inf \{r \in \mathbb{R} : u(x,t)=0\text{ for all }x \geq r\},
\end{equation*}
spread, and converge to the travelling wave $U_{c^*}$ both in shape and speed.

A similar, and even slightly simpler, analysis  to the one for the class of bounded and  compactly supported initial data, allows to prove the following result.

\begin{Theorem}
\label{thm:convergence.class.A}
Let $u$ be a weak solution to~\eqref{eq:main} with $u_0 \in \mathcal{A}$, and let $\zeta$ be the function giving its right free boundary. Then there exists a value $\xi_0 \in \mathbb{R}$ such that
\begin{equation*}
\label{eq:convergence.class.A}
\lim\limits_{t \rightarrow \infty} \sup\limits_{x \in \mathbb{R}} |u(x,t) - U_{c^*}(x - c^* t - \xi_0)|= 0,\qquad\lim\limits_{t\to\infty}\zeta(t)-c^*t=\xi_0.
\end{equation*}
\end{Theorem}

If  $\liminf\limits_{x\to -\infty} u_0(x) \leq 1-\delta$, there is in general no guarantee for the solution to spread. However, if $\liminf\limits_{x\to-\infty} u_0>0$, in certain situations we can guarantee that the solution will be pushed into the class $\mathcal{A}$ in a finite time, so that the above convergence result, Theorem~\ref{thm:convergence.class.A}, holds. Let us give two examples.

If $h$ is such that we are under the conditions for the \textit{hair-trigger effect} to hold, then it is trivial to put a monotone nonincreasing in space function below $u_0$ that is going to grow to 1, hence pushing our solution into class $\mathcal{A}$. This condition depends on $h$.

If $u_0$ is not in class $\mathcal{A}$, but it is big enough at $-\infty$ for us to apply the same ideas of Lemma~\ref{lema_propagation} but with a non-symmetric function
\begin{equation}\label{eq:three.parameters.class.A}
\mathfrak{u}_0(x; \rho,c,\eta)  =
			\begin{cases}
            \eta,\quad &x\leq \rho, \\
            U^{c,\eta}(x - \rho),\quad &\rho < x\leq \rho + b, \\
             0,\quad &x\geq \rho + b,
            \end{cases}
\end{equation}
then $u$ will also be pushed into  $\mathcal{A}$.

\begin{Corollary}
Let $u_0$ be non-negative, bounded, piecewise continuous, $u_0(x) \equiv 0$ for all $x \geq \bar x$ for some  $\bar x \in \mathbb{R}$ and $\liminf\limits_{x\to-\infty} u_0(x)>0$. If the reaction term $h$ satisfies the hypotheses of Theorem~\ref{thm:spreading.vs.vanishing}, or if there exist $x_0\in\mathbb{R}$ and admissible $\rho,c,\eta>0$  such that  $u(x,T)\ge \mathfrak{u}_0(x-x_0; \rho,c,\eta)$ for some $T\geq 0$,  with $\mathfrak{u}_0$ as in \eqref{eq:three.parameters.class.A}, then the conclusions of Theorem~\ref{thm:convergence.class.A} hold.
\end{Corollary}

%



\end{document}